\documentclass[11pt]{article}

\usepackage[top=1in,bottom=1.25in,left=1.25in,right=1.25in]{geometry}

\usepackage{graphicx}
\usepackage{ragged2e}
\usepackage{tabto}
\usepackage{amssymb}
\usepackage{float}
\usepackage{amsmath}
\usepackage{multirow}
\usepackage{amsthm}
\usepackage{epsfig}
\usepackage{subfigure}
\usepackage{amscd}
\usepackage{color}
\usepackage{url}
\usepackage{bigints}
\usepackage[capitalise]{cleveref}
\usepackage{afterpage}
\usepackage{fancybox}
\usepackage{ifthen}
\usepackage{multicol}
\usepackage[font=footnotesize]{caption}

\usepackage{lmodern}
\usepackage{siunitx}
\usepackage[sf,bf,medium]{titlesec}

\usepackage{dsfont}
\usepackage{booktabs}

\newcommand\dist{\operatorname{dist}}

\newcommand\tot{\operatorname{tot}}
\newcommand\In{\operatorname{in}}

\newcommand\cH{\mathcal H}

\newcommand\bc{\boldsymbol c}

\newcommand\btheta{\boldsymbol \theta}

\newcommand\tbB{\widetilde{\boldsymbol B}}

\newcommand\bl{\boldsymbol l}

\newcommand\bm{\boldsymbol m}

\newcommand\bv{\boldsymbol v}
\newcommand\bn{\boldsymbol n}
\newcommand\bt{\boldsymbol t}

\newcommand\bx{\boldsymbol x}
\newcommand\by{\boldsymbol y}

\newcommand\bA{\boldsymbol A}
\newcommand\bB{\boldsymbol B}

\newcommand\bV{\boldsymbol V}
\newcommand\bJ{\boldsymbol J}

\newcommand\bX{\boldsymbol X}

\newcommand\cW{\mathcal W}

\newcommand\cA{\mathcal{A}}
\newcommand\cC{\mathcal{C}}
\newcommand\cL{\mathcal{L}}
\newcommand\cS{\mathcal{S}}

\newcommand\cD{\mathcal{D}}

\newcommand\bbR{\mathbb R}

\newcommand\Npat{N_{\textrm{patches}}}
\newcommand\smax{s_{\textrm{max}}}
\newcommand\nmax{n_{\textrm{max}}}
\newcommand\smin{s_{\textrm{min}}}
\newcommand\niter{n_{\textrm{iter}}}
\newcommand\pa{\partial}

\DeclareMathOperator{\dR}{dR}
\newtheorem{theorem}{Theorem}

\theoremstyle{definition}

\theoremstyle{remark}
\newtheorem{remark}{Remark}

\numberwithin{equation}{section}

\crefname{equation}{}{}

\begin{document}

%\maketitle

\begin{titlepage}

  \raggedleft
  {\texttt{Technical Report\\
    \today}}
  
  \hrulefill

  \vspace{4\baselineskip}

  \raggedright
  {\LARGE \sffamily\bfseries Debye source representations for type-I superconductors, I}
  
  \vspace{\baselineskip}
  {\Large \sffamily The static type I case}

  \vspace{3\baselineskip}
  % author 1
 \vspace{\baselineskip}
 
 \vspace{\baselineskip}
 
  \normalsize Charles L. Epstein \\
  \small \emph{Department of Mathematics, University of Pennsylvania\\
  Philadelphia, PA 19104}\\
  \texttt{cle@math.upenn.edu}
  
 \vspace{\baselineskip}
    \normalsize Manas Rachh \\
    \small \emph{Flatiron Institute, Center for Computational Mathematics \\
    New York, NY 10010}
    
    \texttt{mrachh@flatironinstitute.org}

\end{titlepage}

\tableofcontents
\newpage

In this note, we analyze the classical magneto-static approach to the theory of
\mbox{type I} superconductors, and a Debye source representation that can be used
numerically to solve the resultant equations. We also prove that one of
the fields, $\bJ^-$, found within the superconductor via the London equations, is
the physical current in that the outgoing part of the magnetic field is given as the
Biot-Savart integral of $\bJ^{-}$. Finally, we compute the static currents for moderate values of London penetration depth, $\lambda_L,$ for a sphere, a stellarator-like geometry and a two-holed torus. 
%In a subsequent paper we study the behavior of our equations as $\lambda_L\to 0.$

\section{The Model}
Suppose that the superconducting material is homogeneous, isotropic
and occupies a region $\Omega\subset \bbR^3.$ Within $\Omega$ there is a magnetic
field $\bB^-$, and a current $\bJ^-$, which satisfy the London equations given by:
\begin{equation}\label{eqn1}
  \nabla \times \bB^- = \bJ^- \, ,  \, \quad \text{and} \quad  \nabla \times \bJ^{-} = - \frac{1}{\lambda_L^2} \bB^{-} \,.
\end{equation}
The constant $\lambda_L$ is called the London penetration depth. For physical materials it lies in the $10-200$ nanometer range. A solution to these equations can be expressed, within $\Omega,$ via the Stratton-Chu formula, as an integral over the boundary of $\Omega$ (denoted by $\Gamma$) using the Green's function
\begin{equation}
  g_{\frac{i}{\Lambda_L}}(\bx,\by)=\frac{e^{-\frac{|\bx-\by|}{\lambda_L}}}{4\pi|\bx-\by|}.
\end{equation}
From this formula it follows that the solution falls off like
$e^{-\frac{\dist(\bx,\Gamma)}{\lambda_L}},$ which, for physically
reasonable values of $\lambda_L,$ means that the solutions are
almost zero outside of a very thin neighborhood of $\Gamma.$

Since $\nabla \cdot \nabla \times \bV = 0$ for any vector field $\bV$,~\cref{eqn1}
implies that
\begin{equation}\label{eqn2}
  \nabla \cdot \bB^{-} = \nabla \cdot \bJ^{-} = 0 \, .
\end{equation}
In the complementary region there is a static magnetic field $\bB^{\tot},$ which
satisfies the usual magneto-static equations 
\begin{equation}\label{eqn3}
  \nabla \times \bB^{\tot}=0\text{ and } \nabla \cdot \bB^{\tot}=0\text{ in }\Omega^c.
\end{equation}

While $\bJ^{-}$ is supported in a thin neighborhood of $\Gamma$, there is no current sheet on the boundary of $\Omega,$ and so the physically
reasonable boundary conditions are that
\begin{equation}\label{eqn_bc}
  \bB^{-} = \bB^{\tot} \quad \text{on } \Gamma \, ,
\end{equation}
i.e., the magnetic field is
continuous across the boundary.
 This boundary condition and $\nabla \times \bB^{\tot}=0$  imply that $\bB^{-}$ is a closed vector field on the manifold $\Gamma$ and thus satisfies
\begin{equation}
  \nabla_{\Gamma} \cdot \left( \bn \times\bB^{-} \right) = 0 \quad \text{on } \Gamma \, .
\end{equation}
Here $\nabla_{\Gamma} \cdot$ denotes the surface divergence operator on the boundary $\Gamma$
and $\bn$ is the outward pointing unit normal (pointing into $\Omega^{c}$) to $\Gamma$.
The London equation shows that $\bJ^- \cdot \bn=0$ is equivalent to
$\nabla_{\Gamma} \cdot \bn \times \bB^{-}=0.$ This theory, at least
in the static case, seems to be widely accepted, see~\cite{Cook2}, though it is restricted to circumstances where the induced fields in the superconductor remain below the critical level.

In the standard ``scattering'' problem for this set-up there is an incoming
magneto-static field $\bB^{\In},$ so that, in $\Omega^c,$
$\bB^{\tot}=\bB^{\In}+\bB^+,$ where $\bB^+$ is an outgoing field satisfying the radiation condition
\begin{equation}
  \left|\bB^+(\bx)\right|=o(|\bx|^{-1}).
\end{equation}
The boundary conditions in~\cref{eqn3} are then given by
\begin{equation}
  \bB^-= (\bB^+ + \bB^{\In}) \,  \quad \text{on } \Gamma.
\end{equation}
 If the domain $\Omega$ is non-contractible, then this problem does not have a
 unique solution. If $\Gamma$ has genus $g,$ then there are $g$ A-cycles
 $\{A_1,\dots,A_g\},$ which bound surfaces $\{S_{A_1},\dots,S_{A_g}\}$ contained
 within $\Omega,$ and $g$ B-cycles $\{B_1,\dots,B_g\},$ which bound surfaces
 $\{S_{B_1},\dots,S_{B_g}\}$ contained in $\Omega^c,$ see~\cref{fig:tor} for the genus 1 case.

\begin{figure}[h!]
    \centering
    \includegraphics[width=0.7\linewidth]{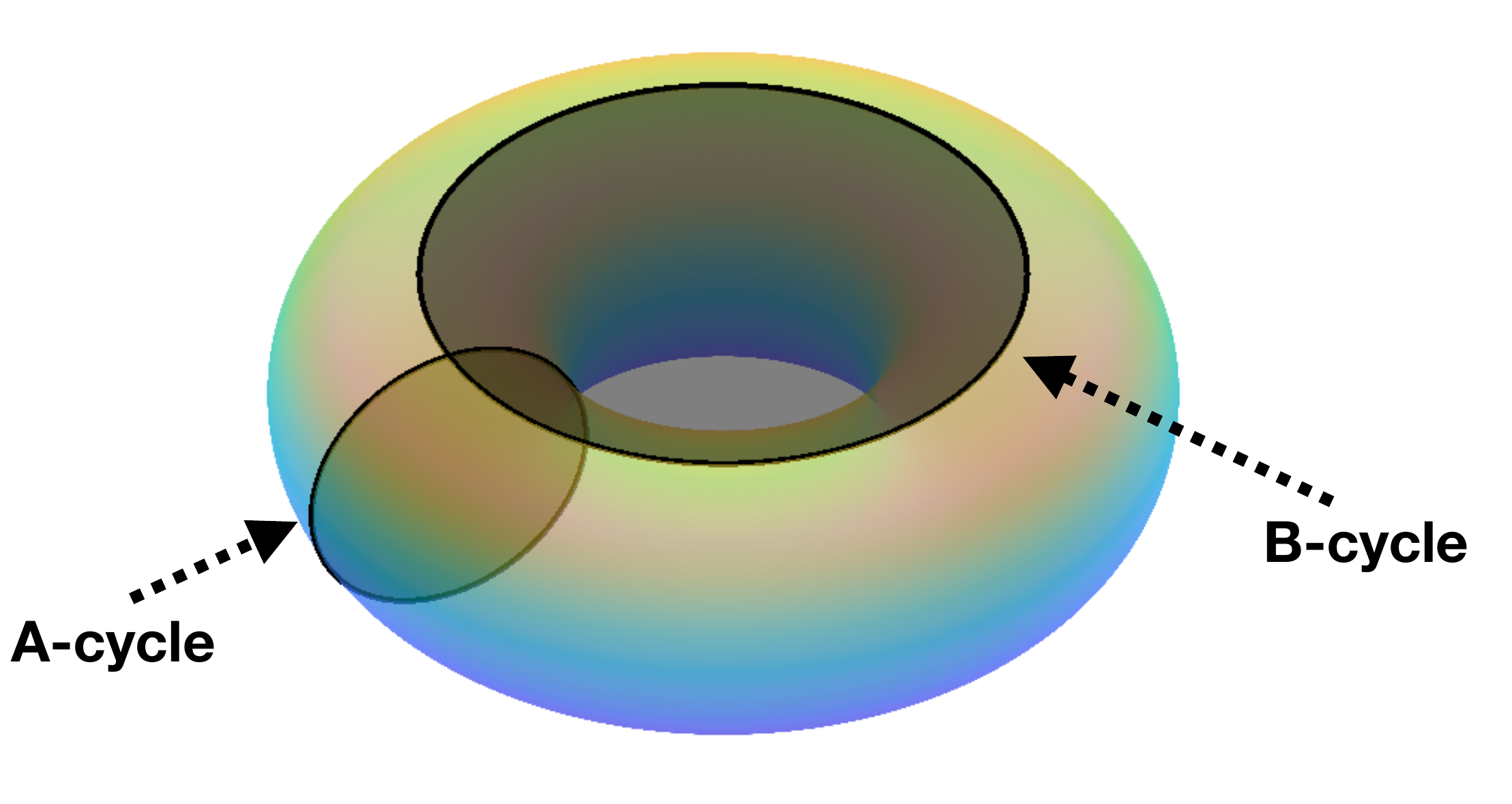}
    \caption{A-cycle and B-cycle, with spanning surfaces,  for a genus 1 torus.}
    \label{fig:tor}
\end{figure}

 Since $-\bn \times \bn \times \bB$ is a closed vector field on $\Gamma$, its integrals over
 cycles on $\Gamma$ depend only on the homology class of the cycle. It
 follows from the boundary condition, and Stokes theorem that
 \begin{equation}
   \int_{B_j} \bB^{-} \cdot d\ell= \int_{B_j} (\bB^{+} + \bB^{\In}) \cdot d\ell
   = \int_{S_{B_j}} \nabla \times (\bB^{+} + \bB^{\In}) \cdot \bn \,  dS = 0.
 \end{equation}
 On the other hand the fluxes
 \begin{equation}
   a_j=\int_{A_j}\bB^- \cdot d\ell = \int_{S_{A_j}}\nabla \times \bB^{-} \cdot \bn  \, dS=
   \int_{S_{A_j}}\bJ^{-} \cdot \bn \, dS,\text{ for }j=1,\dots, g,
 \end{equation}
 are not determined {\em a priori}, and in fact, constitute additional data that must
 be specified to get a uniqueness theorem.

 \section{The Uniqueness Theorem \label{sec:uniqueness}}
 Before proceeding to the uniqueness theorem we consider a minimal
 parametrization for solutions to this system of equations. The bounded domain
 $\Omega$ always has a trivial 2-dimensional cohomology group, and therefore
 $\bB^-=\nabla \times \bA^{-},$ with $\bA^-$ a solution of
 \begin{equation}\label{eqn9}
   -\Delta \bA^-=-\frac{1}{\lambda_L^2}\bA^-.
 \end{equation}
 We then define $\bJ^-=\nabla (\nabla \cdot \bA^-)-\frac{1}{\lambda_L^2}\bA^-.$ With these
 choices, $\bB^-$ and $\bJ^-$ satisfy the equations~\eqref{eqn1} and
 \eqref{eqn2} . The set of vector fields satisfying~\eqref{eqn9} depend on 3 functions
 defined on $\Gamma,$ for example $-\bn \times (\bn \times \bA^{-})$ and
 $-\nabla \cdot \bA^-.$ Using the gauge freedom
 $\bA^-\rightarrow \bA^-+\nabla \phi^-,$ where
 $\lambda_L^2\Delta\phi^--\phi^-=0,$ we can reduce this to two functions on the
 boundary.
 
There are $g$ outgoing harmonic vector fields defined
 in $\Omega^c,$ denoted by 
 $\{\bB^H_1,\dots,\bB^H_g\},$ which satisfy $\nabla \times \bB^{H}_{j} = \nabla \cdot \bB^{H}_j = 0$ for $\bx \in \Omega^{c}$, and
 $\bB^{H}_j \cdot \bn =0,$ for $\bx \in \Gamma$, $j=1,2,\ldots g.$ An arbitrary vector field
 $\bB^+,$ satisfying $\nabla \times \bB^{+}=\nabla \cdot \bB^+=0$ has a unique representation as
 \begin{equation}
   \bB^+=\nabla f^++\sum_{j=1}^gh_j\bB^H_j,
 \end{equation}
 where $\Delta f^+=0.$ Thus outgoing solutions to magneto-static problems depend on 1-function of mean 0 on
 $\Gamma$ ($\pa_{\bn}f^+$), and the $g$ real numbers $\{h_1,\dots,h_g\}.$

 We  now state and prove a uniqueness theorem for this problem.
 \begin{theorem}
   Suppose that $(\bB^-,\bJ^-),$ and $\bB^+$ satisfy~\eqref{eqn1}
   and~\eqref{eqn3} respectively and the homogeneous boundary conditions
   \begin{equation}
        \bB^{-} = \bB^{+} \quad \text{on } \Gamma \, ,
   \end{equation}
   along with the conditions
   \begin{equation}
     \int_{A_j}\bB^- \cdot d\ell=0 \, \, \textrm{for } j=1,2,\ldots , g.
   \end{equation}
   If, in addition, $\bB^+$ is an outgoing solution, then $(\bB^-,\bJ^-)\equiv
   (0,0)$ in $\Omega$ and $\bB^+\equiv 0,$ in $\Omega^{c}$.
 \end{theorem}
 \begin{proof}
   We use integration parts. Using equation~\eqref{eqn1} and Stokes theorem we
   see that 
   \begin{equation}
     \begin{split}
     \int_{\Omega}\left[\frac{1}{\lambda_L^2}| \bB^{-} |^2+|\bJ^{-}|^2 \right] dV&=
     \int_{\Gamma} \bn \times \bB^{-} \cdot \bJ^{-} dS.
       \end{split}
   \end{equation}
   The hypotheses of the theorem imply that $\bB^+$ has no projection onto
   the harmonic vector fields $\bB^H_{j}$, $j=1,2,\ldots g$, and thus there exists a harmonic function $f$ defined 
 in $\Omega^c$ such that $\bB^{+} = \nabla f.$ Therefore, as
   $\bB^+$ is outgoing,
   \begin{equation}
     \int_{\Omega^c}|\bB^{+}|^2 dV= \int_{\Omega^c} |\nabla f|^2 dV = -\int_{\Gamma} f \bB^{+} \cdot \bn dS \, ,
   \end{equation}
   Note the extra negative sign since the normal is pointing into $\Omega^c$.
   Because $\bB^{-} = \bB^{+}$ on the boundary, 
   we see that
   \begin{equation}
     \begin{split}
     \int_{\Gamma}  \bn \times \bB^{-} \cdot \bJ^{-} &= 
     \int_{\Gamma} \bn \times \nabla f \cdot \bJ^{-} dS\\
       &=\int_{\Gamma}[-f \nabla \times \bJ^- \cdot \bn + \nabla \times (f \bJ^{-}) \cdot \bn ]\\
       &=\frac{1}{\lambda_L^2}\int_{\Gamma} f\bB^+ \cdot \bn.
     \end{split}
   \end{equation}
   Combining these various computations we see that
   \begin{equation}
    \int_{\Omega}\left[\frac{1}{\lambda_L^2}| \bB^{-} |^2+|\bJ^{-}|^2 \right] dV
     +\frac{1}{\lambda_L^2}\int_{\Omega^c}|\bB^{+}|^2 dV=0,
   \end{equation}
   from which the conclusion is immediate.
 \end{proof}
 The theorem shows that $\bB^{\In}(\bx)$ for $\bx$  on  $\Gamma$ along
 with the $g$ real parameters $\{a_1,\dots,a_g\},$ uniquely determine
 a solution to the ``scattering'' problem for a superconductor.  Note
 that this theorem leaves open the possibility that there is
 $g$-parameter family of purely outgoing solutions.  In the next
 section we will see that such solutions do in fact exist. In the
 physically important toroidal case, this theorem implies that the
 static current distribution in a superconducting torus is determined
 by its flux across $S_A.$

 This uniqueness result, in the simply connected case, appears in~\cite{Cook2}.

\section{Layer potentials \label{sec:layerpot}}
In this section we describe some elementary properties of scalar and vector layer potentials for the Helmholtz equation. These results are classical and can be found in~\cite{ColtonKress} for example.
Suppose that $g_{k}(\bx,\by)$ is the Green's function for the scalar Helmholtz equation with wave number $k$ given by
\begin{equation}
g_{k}(\bx,\by) = \frac{e^{ik|\bx-\by|}}{4\pi |\bx-\by|}
\end{equation}
Here $k$ is any complex number.
Let 
$\cS_{k} [\sigma]$, $\cD_{k}[\sigma]$ denote the Helmholtz single and double layer potentials given by
\begin{equation}
\cS_{k}[\sigma](\bx) = \int_{\Gamma} g_{k}(\bx,\by) \sigma(\by) \, dS  \, , \quad  
\cD_{k}[\sigma](\bx) = \int_{\Gamma} \left( \bn(\by) \cdot \nabla_{\by} g_{k}(\bx,\by)  \right)\sigma(\by) \, dS \, .
\end{equation}
Let $\cS_{k}'[\sigma]$ and $\cD_{k}'[\sigma]$ denote the principal value and finite part of the Neumann data corresponding to the single and double layer potentials given by
\begin{equation}
\begin{aligned}
\cS_{k}'[\sigma](\bx) &={\rm p.v.} \left(\bn(\bx) \cdot \nabla_{\bx} \left( \int_{\Gamma} g_{k}(\bx,\by) \sigma(\by) \, dS  \right)  \right)  \, , \\
\cD_{k}'[\sigma](\bx) &= {\rm f.p.}\left( \bn(\bx) \cdot \nabla_{\bx} \left( \int_{\Gamma} \left( \bn(\by) \cdot \nabla_{\by} g_{k}(\bx,\by)  \right)\sigma(\by) \, dS  \right) \right) \, .
\end{aligned}
\end{equation}
$\cS_{k}''$ denotes the finite part of the second normal derivative of $\cS_{k}$ restricted to $\Gamma$ given by
\begin{equation}
\cS_{k}''[\sigma](\bx) ={\rm f.p.} \left(\bn(\bx) \cdot \nabla_{\bx} \nabla_{\bx }\left( \int_{\Gamma} g_{k}(\bx,\by) \sigma(\by) \, dS  \right) \cdot \bn(\bx) \,  \right) \, , \\
\end{equation}
and let $H(\bx)$ denote the mean curvature for $\bx \in \Gamma$. The abbreviation ``p.v.'' denotes a principal value integral and the abbreviation ``f.p.'' denotes the finite part of the integral.

The layer potentials satisfy the following jump relations and Calderon identities.
Suppose that $\bx_{0} \in \Gamma$, then 
\begin{equation}
\label{eq:identites}
\begin{aligned}
\lim_{\substack{\bx \to  \bx_{0} \\ \bx \in \Omega}} \nabla_{\bx} \cS_{k} [\sigma] &= -\frac{\sigma(\bx_{0}) \bn(\bx_{0})}{2} + {\rm p.v.} \nabla_{\bx} \cS_{k}[\sigma](\bx_{0}) \, ,\\
\lim_{\substack{\bx \to \bx_{0} \\ \bx \in \Omega}} \nabla_{\bx} \times \cS_{k} [\bv] &= \frac{\bn(\bx_{0}) \times \bv(\bx_{0})}{2} + {\rm p.v.} \left(\nabla_{\bx} \times \cS_{k}[\bv](\bx_{0}) \right) \, ,\\
\lim_{\substack{\bx \to \bx_{0} \\ \bx \in \Omega^c}} \bn(\bx_{0}) \cdot \nabla_{\bx} \cS_{k} [\sigma] &= \frac{\sigma(\bx_{0})}{2} + \cS_{k}'[\sigma](\bx_{0}) \, , \\
\lim_{\substack{\bx \to \bx_{0} \\ \bx \in \Omega^{c}}} \nabla_{\bx} \times \cS_{k} [\bv] &= -\frac{\bn(\bx_{0}) \times \bv(\bx_{0})}{2} + {\rm p.v.} \left(\nabla_{\bx} \times \cS_{k}[\bv](\bx_{0}) \right) \, ,\\
\cS_{0} \Delta_{\Gamma} \cS_{0}[\sigma](\bx_{0}) &= -\frac{\sigma(\bx_{0})}{4} + \left( \cD_{0}^2 - \cS_{0} \left( \cS_{0}'' + \cD_{0}' - 2H \cS_{0}' \right)  \right)[\sigma](\bx_{0}) \, , \\
\Delta_{\Gamma} \cS_{0}^2[\sigma](\bx_{0}) &= -\frac{\sigma(\bx_{0})}{4} + \left( (\cS_{0}')^2 - \left( \cS_{0}'' + \cD_{0}' - 2H \cS_{0}' \right)  \cS_{0} \right)[\sigma](\bx_{0}) \, .
\end{aligned}
\end{equation}
In these equations, $\sigma$ is a scalar function, and $\bv$ a vector field defined on $\Gamma.$

\begin{remark}
Even though $\cS''_{k}$ and $\cD'_{k}$ are finite part integrals, in all integral equations, the two terms always appear as a sum $\cS_{k}'' + \cD_{k}'$ which is in fact a compact operator of order $-1$.
\end{remark}

 \section{The Debye Source Representation}
 To define the Debye source representations for $(\bB^-,\bJ^-)$ and $\bB^+$
 it is useful to first rescale the fields inside of $\Omega$ by setting
 $$\tbB^-=\frac{1}{\lambda_L}\bB^-.$$
 The rescaled fields, then satisfy the more symmetric system of equations
 \begin{equation}
   \nabla \times \tbB^-=\frac{1}{\lambda_L}\bJ^- \text{ and } \nabla \times \bJ^- =-\frac{1}{\lambda_L}\tbB^-.
 \end{equation}
 The rescaled fields have the following Debye source representation
\begin{equation}\label{eqn18}
   \begin{split}
     \tbB^-=-\frac{1}{\lambda_L}\btheta^- + \nabla \Psi^-+\nabla \times \bA^- \,& , \quad
     \bJ^-=-\frac{1}{\lambda_L}\bA^--\nabla\phi^--\nabla \times \btheta^- \, , \text{ and}\\
     \bB^+=\nabla \times \bA^+&+\nabla \Psi^+ \, ,
   \end{split}
\end{equation}
where $\Psi^{\pm},\phi^-$ are scalar functions, and $\bA^{\pm},\btheta^-$ are vector fields.
In order for these fields to satisfy the relevant PDEs it is necessary that
\begin{equation}\label{eqn19}
  \nabla \cdot \bA^-=-\frac{1}{\lambda_L}\phi^-,\quad \nabla \cdot \btheta^-=\frac{1}{\lambda_L}\Psi^- \, ,
\end{equation}
and $\nabla \cdot \bA^+=\Delta \Psi^+=0.$ These latter conditions are automatic for
\begin{equation}
\label{eq:bpp}
  \Psi^+=\cS_{0}[q^+] \text{ and }\bA^+= \cS_{0}[\bl^{+}_{H}] \, ,
\end{equation}
where $\cS_{0}$ is the single layer potential defined in~\cref{sec:layerpot}, and $\bl^+_H$ is a harmonic vector field on $\Gamma$.

Let $k=i/\lambda_{L}$. For the interior fields we set
\begin{equation}\label{eqn22}
    \bA^-= \cS_{k}[\bl^{-}] \, , \quad \btheta^{-} = \cS_{k}[\bm^{-}] \, , \quad \phi^{-} =\cS_{k}[r^{-}] \, , \quad \text{and} \, ,\quad
    \Psi^{-} = \cS_{k}[q^{-}] \, ,
\end{equation}
where $\bl^{-}$ and $\bm^{-}$ are vector fields, which are tangent to $\Gamma.$
For the equations in~\eqref{eqn19} to hold it is only necessary that
\begin{equation}
\nabla_{\Gamma} \cdot \bl^-=-\frac{1}{\lambda_L}r^-\text{ and }
\nabla_{\Gamma} \cdot \bm^- =\frac{1}{\lambda_L}q^- \, ,
\end{equation}
where $\nabla_{\Gamma} \cdot $ is the surface divergence operator.
These relations force $q^{-}$ and $r^{-}$ to be functions of mean zero on $\Gamma$. As we show below, this is not generally true of $q^{+}$.
With this in mind we set
\begin{equation}
 \begin{split}
  \bm^- &=\frac{1}{\lambda_L}\nabla_{\Gamma} \Delta_{\Gamma}^{-1}q^- \, ,\\
  \bl^-&=-\frac{1}{\lambda_L}\left[\nabla_{\Gamma} \Delta_{\Gamma}^{-1}r^- +\bl_H^- \right] \, .
  \end{split}
\end{equation}
Here $\nabla_{\Gamma}$ is the surface gradient operator, which is the adjoint of the
surface divergence operator on the space of $L^2$ functions with respect to the induced 
metric on the surface; $\Delta_{\Gamma}$ is the Laplace Beltrami operator on $\Gamma$; 
and $\bl_H^-$ is a harmonic vector field on $\Gamma.$ 
The inverse of $\Delta_{\Gamma}$ should be interpreted as inverse on functions of mean zero.
This is essentially what is
done in~\cite{EpGr2}, and therefore we can follow the computations there  to work out
the form of the boundary integral equations.

Using the jump relations and identities for layer potentials discussed in~\cref{sec:layerpot}, the limiting values of $\tbB^{-}$, $\bB^{+}$ and $\bJ^{-}$ for $\bx \in \Gamma$ are given by
\begin{equation}
\begin{aligned}
\tbB^{-} &= \frac{q^{-} \bn + \bn \times \bl^{-}}{2}  -\frac{1}{\lambda_L} \cS_{k} [\bm^{-}]  + {\rm p.v.} \left( \nabla \cS_{k}[q^{-}] + \nabla \times \cS_{k}[\bl^{-}] \right) \, ,\\ 
\bB^{+} &= -\frac{q^{+} \bn + \bn \times \bl^{+}_{H} }{2} + {\rm p.v.} \left( \nabla \cS_{0}[q^{+}] + \nabla \times \cS_{0} [\bl_{H}^{+}] \right) \, , \\
\bJ^{-} &= -\frac{r^{-} \bn + \bn \times \bm^{-}}{2} -\frac{1}{\lambda_{L}} \cS_{k}[\bl^{-}] - 
{\rm p.v.} \left( \nabla \cS_{k}[r^{-}] + \nabla \times \cS_{k}[\bm^{-}] \right)  \, .
\end{aligned}
\end{equation}

For the boundary conditions we start with
\begin{equation}\label{eqn25}
  \begin{split}
    -\lambda_L \bn \times (\bn \times \tbB^-)&=-\bn \times (\bn \times [\bB^++\bB^{\In}])\quad \text{on } \Gamma \, ,\\
    \lambda_L\tbB^- \cdot \bn &=[\bB^++\bB^{\In}] \cdot \bn \quad \text{on }\Gamma \, .
    \end{split}
\end{equation}
Following~\cite{EpGr2} we apply $\cS_0 \nabla_{\Gamma} \cdot$ to both sides of the first equation
in~\eqref{eqn25} to obtain the following equation
\begin{equation}\label{eqn26}
\begin{aligned}
\frac{-\lambda_{L} q^{-} + q^{+}}{4} + \lambda_{L} \cS_{0} \Delta_{\Gamma} \left( \cS_{k} - \cS_{0}\right) [q^{-}] + 
\left( \cD_{0}^2  -\cS_{0} \left( \cS_{0}'' + \cD_{0}' - 2H\cS_{0}\right) \right) [\lambda_{L} q^{-} - q^{+}] &+ \\
 \cS_{0} \left[\nabla_{\Gamma} \cdot \bn \times \bn \cdot \left( \cS_{k}[\bm^{-}]- \lambda_{L} \nabla \times \cS_{k}[\bl^{-}] + \nabla \times \cS_{0}[\bl_{H}^{+}] \right) \right] = -\cS_{0}[\nabla_{\Gamma} \cdot \bn \times \bn \times& \bB^{\In} ]  \, . \\
\end{aligned}
\end{equation}
This follows from the fact the limiting values of $\cS_{0} [\nabla_{\Gamma} \cdot ]$ applied to the tangential components of the interior and exterior magnetic fields are given by
\begin{equation}
\begin{aligned}
-\cS_{0}[\nabla_{\Gamma} \cdot \bn \times \bn \times \tbB^{-}] &= \cS_{0} \Delta_{\Gamma} \cS_{k} [q^{-}] + \cS_{0} \left[ \nabla_{\Gamma} \cdot \bn \times \bn \times \left( \frac{1}{\lambda_L}\cS_{k}[\bm^{-}] - \nabla \times \cS_{k}[\bl^{-}] \right) \right] \\
&= -\frac{q^{-}}{4} + \cS_{0} \Delta_{\Gamma} \left(\cS_{k}- \cS_{0} \right) [q^{-}] + 
\left( \cD_{0}^2  -\cS_{0} \left( \cS_{0}'' + \cD_{0}' - 2H\cS_{0}\right) \right) [q^{-}] + \\
& \quad \quad \quad \cS_{0} \left[ \nabla_{\Gamma} \cdot \bn \times \bn \times \left( \frac{1}{\lambda_L}\cS_{k}[\bm^{-}] - \nabla \times \cS_{k}[\bl^{-}] \right) \right]  \\
-\cS_{0}[\nabla_{\Gamma} \cdot \bn \times \bn \times \bB^{+}] &= \cS_{0} \Delta_{\Gamma} \cS_{0} [q^{+}] - \cS_{0} \left[ \nabla_{\Gamma} \cdot \bn \times \bn \times \nabla \times \cS_{0}[\bl_{H}^{+}] \right] \\
&= -\frac{q^{+}}{4} + 
\left( \cD_{0}^2  -\cS_{0} \left( \cS_{0}'' + \cD_{0}' - 2H\cS_{0}\right) \right) [q^{+}] - \\
 &\quad \quad \quad \cS_{0} \left[ \nabla_{\Gamma} \cdot \bn \times \bn \times \nabla \times \cS_{0}[\bl_{H}^{+}] \right]  \, .
\end{aligned}
\end{equation}
The second equation in~\cref{eqn25} takes the form
\begin{equation}\label{eqn27}
 \frac{\lambda_{L} q^{-} + q^{+}}{2} + \lambda_{L} \cS_{k}'[q^{+}] - \cS_{0}'[q^{+}] +  \bn \cdot \left( - \cS_{k}[\bm^{-}]+ \lambda_{L} \nabla \times \cS_{k}[\bl^{-}] - \nabla \times \cS_{0}[\bl_{H}^{+}] \right) = \bn\cdot \bB^{\In}  \\
\end{equation}
The second condition in~\cref{eqn25} also implies that
\begin{equation}
\int_{\Gamma} \left(\lambda_{L} \tbB^{-} - \bB^{+} \right) \cdot \bn \, dS = \int_{\Gamma} \bB^{\In} \cdot \bn \, dS  \, .
\end{equation}
In terms of the Debye sources we see that
\begin{equation}
\begin{aligned}
\int_{\Gamma} \lambda_{L} \tbB^{-} \cdot \bn \,  dS &= \int_{\Gamma} \left(-\btheta^{-} + \lambda_{L} \nabla \Psi^{-} + \lambda_{L}\nabla \times \bA^- \right) \cdot \bn \,  dS \\
 &= \int_{\Omega}\left( -\nabla \cdot \btheta^{-} + \lambda_{L} \Delta \Psi^{-} \right) \, dV = 0 \, , \quad \text{and} \\
\int_{\Gamma} -\bB^{+} \cdot \bn \, dS &= -\int_{\Gamma} \left(\nabla \times \bA^{+} + \nabla \Psi^{+} \right) \cdot \bn \, dS = 
\int_{\Gamma} q^{+} \, dS \, .
\end{aligned} 
\end{equation}
The last equality follows by using the jump relation for the integral defining $\nabla \Psi^{+}$, and the fact that it is a closed vector field in $\Gamma^{c}$. Thus, we see that
\begin{equation}
\int_{\Gamma} q^{+} dS = \int_{\Gamma} \bB^{\In} \cdot \bn \, dS \,.
\end{equation}

  To obtain the final equation we use instead the condition that $\bJ^{-} \cdot \bn =0.$
  Since $\lambda_L \nabla \times \tbB^{-}=\bJ^-,$ this will imply that
  $\nabla_{\Gamma} \cdot \bn \times \tbB^{-} = \bJ^{-} \cdot \bn / \lambda_{L} =0$. In the Debye representation the condition $\bJ^- \cdot \bn=0$ takes the form:
    \begin{equation}\label{eqn28}
-\frac{r^{-}}{2} -\cS_{k}'[r^{-}] - \bn \cdot \left( \cS_{k}[\bm^{-}] + \nabla \times \cS_{k} [\bl^{-}] \right) =0 \, , \\
    \end{equation}
    It is clear that equations~\eqref{eqn26},~\eqref{eqn27}, and~\eqref{eqn28} are a Fredholm system of second kind for $(q^-,q^+,r^-).$ These together with the PDEs will
    enforce the conditions that the normal components of total $\bB$-fields
    are continuous across $\Gamma,$ as are the topologically trivial
    components of the Hodge decomposition of the tangential components. This
    still leaves the harmonic components, which are in turn related to the two
    harmonic vector fields $(\bl_H^+,\bl_H^-).$

    The harmonic vector field $\bl_H^-$ is, {\em a priori}, arbitrary, belonging to a
    $2g$-dimensional space. We first need to impose the conditions
    \begin{equation}\label{eqn29}
      \int_{B_j}\tbB^{-} \cdot d\ell=0\text{ for }j=1,\dots,g,
    \end{equation}
    since $\bB$ is continuous across the boundary and $\bB^{+}$ satisfies 
    $\nabla \times \bB^{+} = 0$ in $\Omega^c.$

  The harmonic vector field
    $\bl_H^+$ is selected from a $g-$dimensional subspace 
    $V_+\subset \textrm{ker} \nabla_{\Gamma} \cdot ,$ complementary to common null-space of the
    linear functionals defined by the flux of the associated vector fields $\nabla \times \bA^{+}$ over the A-cycles.
    Thus if $\bl^+_H\in V_+$ and
    \begin{equation}
    \label{eqn:hvec1}
      \int_{B_j} \left(\nabla \times \bA^{+} \right) \cdot d\ell =0\text{ for }=1,\dots,g,
      \text{ then }\bl^+_H=0.
    \end{equation}
   Here $\nabla \times \bA^{+}$ for $\bx$ on $\Gamma$ is defined as the limit from $\Omega^{c}$. These integrals are all zero if the limit is taken in~\cref{eq:bpp} is taken from $\Omega$; hence, using the jump relations for $\nabla \times \bA^{+}$, we see that
   \begin{equation}
   \int_{A_{j}} \nabla \times \bA^{+} \cdot d\ell = -\int_{A_{j}} \bn \times \bl^{+}_{H} \cdot d\ell  \, .
   \end{equation}
Hence for $V_{+}$ we can take the $g-$dimensional subspace of $\bl^{+}_{H} \in \cH^{1}(\Gamma)$ such that integrals of $\bn \times \bl^{+}_{H} \cdot d\ell$ over the $B-$cycles vanish. 

 With this understood, we augment the Fredholm system above with
 the following $4g$ conditions
    \begin{equation}\label{eqn30}
    \begin{aligned}
      \int_{A_j}\lambda_L\tbB^- \cdot d\ell=\int_{A_j}[\bB^++\bB^{\In}] \cdot d\ell=a_j \, ,\\
      \int_{B_{j}}\tbB^{-} \cdot d\ell  = \int_{B_{j}} \bn \times \bl_{H}^{+} \cdot d\ell = 0 \, ,
       \end{aligned}
    \end{equation}
    for $j=1,2,\ldots g$.
    These additional conditions ensure that the topologically non-trivial parts
    of the Hodge decompositions of
    $-\bn \times \bn \times [\bB^++\bB^{\In}]$ and
    $-\bn \times \bn \times \lambda_L\tbB^-$ agree. Taken together, these equations imply
    the continuity of the $\bB$-fields across $\Gamma.$
    
 \begin{remark}
 \label{rem:rep}
There are other choices that can be made in the definition of the sources $\bm^{-}, \bl^{-},$ in terms of the scalar
sources and harmonic vector fields. Indeed, we can let   
\begin{equation}
 \begin{split}
  \bm^- &=\frac{1}{\lambda_L}\left[ \nabla_{\Gamma} \Delta_{\Gamma}^{-1}q^-
  +\sigma_{m} \bn \times \nabla_{\Gamma} \Delta_{\Gamma}^{-1} \left(q^{+} - \frac{1}{|\Gamma|} \int_{\Gamma} q^{+} \right) + 
  \bm_{H}^{-}\right]\\
  \bl^-&=-\frac{1}{\lambda_L}\left[\nabla_{\Gamma} \Delta_{\Gamma}^{-1}r^- 
  + \sigma_{\ell} \bn \times \nabla_{\Gamma} \Delta_{\Gamma}^{-1} \left(q^{+} - \frac{1}{|\Gamma|} \int_{\Gamma} q^{+} \right) + 
  \bl_H^- \right] \, , 
  \end{split}
\end{equation}
with $\sigma_{\ell},\sigma_{m}$ arbitrary constants. The harmonic vector fields $\bm_{H}^{-}$ and $\bl_{H}^{-}$ need to satisfy the condition
\begin{equation}
\int_{\Gamma} \bn \times \bl_{H}^{-} \cdot \bm_{H}^{-} \, dS= 0\, .
\end{equation}
This means that either we fix one of the sources to vanish and then the other is an arbitrary harmonic vector field, or we can select a Lagrangian subspace $\cL$ which is a subset of the harmonic vector fields and both $\bm_{H}^{-}, \bl_{H}^{-} \in \cL$. A subspace, $\cL$, is Lagrangian if its dimension is half the dimension of space of harmonic vector fields and any pair $\bl_1, \bl_2 \in \cL$ satisfy
\begin{equation}
\int_{\Gamma} \bn \times \bl_1 \cdot \bl_2 \, dS= 0\, .
\end{equation}
In~\cref{subsec:sphere}, we briefly explore the effect that $\sigma_{m},\sigma_{\ell}$ have on the conditioning of the resultant system of integral equations.
 \end{remark}   
    
    \section{Injectivity of the Debye Source Representation}
    We have shown that the Debye source representation results in a system of
    Fredholm equations of second kind, but we still need to show that this
    representation has a trivial nullspace. That is we need to show that if we
    are given data $(q^-,r^-,\bl_H^-,q^+,\bl_H^+)$ so that the
    equations~\eqref{eqn26}--\eqref{eqn30} are all satisfied with right hand
    sides equal to zero,  then the data itself must be zero. In
    this context we will assume that the incoming data $\bB^{\In}$ satisfies
    \begin{equation}
      \int_{A_j}\bB^{\In} \cdot d\ell=0\text{ for }j=1,\dots,g.
    \end{equation}
    This can always be arranged by adding an outgoing harmonic vector field to $\bB^{\In}.$

    It follows from these assumptions that for $\bx$ on $\Gamma$, $\bB^{\In} = 0,$
    therefore the fields
    $\bB^{\pm}$ satisfy the homogeneous conditions
    \begin{equation}
      \begin{split}
        \bB^{+} = \bB^{-} &\text{ on } \Gamma\\
        \int_{A_j} \bB^{-} \cdot d\ell =0&\text{ for }j=1,\dots,g.
        \end{split}
    \end{equation}
    Since $\bB^+$ is outgoing, we proved above that these conditions imply that
    $(\bJ^{-},\bB^-)=(0,0)$ in $\Omega,$ and $\bB^+=0$ in $\Omega^c.$

    To show that the Debye representation is injective, we first consider $\bB^+=\nabla \times \bA^++ \nabla \Psi^{+}$ in $\Omega^c.$ The fact that
    this field vanishes shows that
    \begin{equation}
      \int_{A_j}\nabla \times \bA^+ \cdot d\ell=0,\text{ for }j=1,\dots,g.
    \end{equation}
    Recall that
    \begin{equation}
      \bA^+=\int_{\Gamma}g_0(\bx,\by)\bl_H^+(\by) dS,
    \end{equation}
    where $\bl_H^+\in V_+.$ The subspace $V_+
    \subset \cH^1(\Gamma)$ is defined so that the conditions
    satisfied by $ \nabla \times \bA^+$ imply that $\bl_H^+\equiv0,$
    see~\eqref{eqn:hvec1}. Thus we see
    that $\nabla \Psi^{+}=0$ for $x\in\Omega^c,$ where
    \begin{equation}
     \Psi^+(\bx)=\int_{\Gamma}g_0(\bx,\by)q^+(\by)dS.
    \end{equation}
    Thus $\Psi^+(\bx)$ is constant, for $\bx\in\Omega^c,$ and tends to zero at infinity,
    it is therefore identically zero in $\Omega^c.$ The function $\Psi^+(\bx)$ is also
    harmonic in $\Omega$ and continuous across $\Gamma,$ so therefore
    identically zero in $\Omega$ as well. The jump formula for $\pa_{\bn}\Psi^+$
    across $\Gamma$ shows that $q^+\equiv 0.$ This leaves the sources
    $(q^-,r^-,\bl_H^-).$

    The fields $(\bJ^-,\bB^-)$ vanish in $\Omega;$ we let
    $(\bJ^-_1,\bB^-_1)$ denote the fields in $\Omega^c$ defined by these
    sources via~\eqref{eqn18} and~\eqref{eqn22}.  The jump relations for the
    Debye representation,
    \begin{equation}
      \begin{split}
       -\bn\times \bn \times (\bJ_{1}^- - \bJ^-)&=\bn \times \bm^- \, ,\\
       -\bn \times \bn \times (\bB^-_1-\bB^-)&=-\lambda_L \bn \times \bl^-,
      \end{split}
    \end{equation}
    imply that
    \begin{equation}
      -\bn \times \bn \times \bJ^-_1=\bn \times \bm^-\text{ and }
       -\bn \times \bn \times \bB^-_1=-\lambda_L\bn\times \bl^-.
    \end{equation}
    To show that the sources are  zero, we let $\Omega^c_R=\Omega^c\cap B_R,$ and
    observe that the equations satisfied by $(\bJ^-_1,\bB^-_1)$ imply that
    \begin{equation}\label{eqn38}
      \int_{\Omega^c_R} \nabla \cdot (\bJ_1^- \times \bB_1^-)=-\int_{\Omega^c_R}\left[\frac{1}{\lambda_L^2}|\bB^-_1|^2 +
        |\bJ^-_1|^2\right].
    \end{equation}
    On the other hand,  Stokes theorem implies that
    \begin{equation}\label{eqn5.9.5}
      \begin{split}
         \int_{\Omega^c_R}\nabla \cdot (\bJ^-_1\times \bB^-_1)&=
         \int_{\pa\Omega^c_R} (\bJ^-_1 \times \bB^-_1) \cdot \bn dS \\
         &= \lambda_L\int_{\Gamma}(\bm^-\times \bl^-) \cdot \bn+\int_{\pa B_R}(\bJ^-_1\times \bB^-_1) \cdot \bn dS\\
         &= -\int_{\Gamma}(\nabla_{\Gamma} \Delta_{\Gamma}^{-1}q^- \times \bl^-) \cdot \bn+\int_{\pa B_R}\bJ^-_1\times \bB^-_1 \cdot \bn  \\
          &= -\int_{\Gamma}\nabla_{\Gamma} \cdot \bn \times (\Delta_{\Gamma}^{-1}q^- \bl^-) + \int_{\Gamma} \Delta_{\Gamma}^{-1}q^- \nabla_{\Gamma} \cdot (\bn \times \bl^-)+
          \int_{\pa B_R}\bJ^-_1\times \bB^-_1 \cdot \bn \, .\\
      \end{split}
    \end{equation}
    As $\lambda_L\bm^-=
    \Delta_{\Gamma}^{-1}q^-$. The first term in the last line of~\eqref{eqn5.9.5} is 0 by Stokes theorem on $\Gamma,$ the second term is zero since $\nabla_{\Gamma} \cdot \bn \times \bl^-=0$, and the outgoing condition implies that
    \begin{equation}
      \lim_{R\to\infty}\int_{\pa B_R}\bJ^-_1\times \bB^-_1\geq 0,
    \end{equation}
    which, along with~\eqref{eqn38}, shows that $(\bJ^-_1,\bB^-_1)$ also
    vanish in $\Omega^c.$ From the jump relations we see that $\Delta_{\Gamma}^{-1}q^-$ is
    constant, and therefore 0. Since the components of the Hodge decomposition are orthogonal this also implies that $\Delta_{\Gamma}^{-1}r^-=0$ and  $\bl_H^-=0.$ 
    Since $\Delta_{\Gamma}$ is injective on functions of mean zero we see that $q^-=r^-=0,$
    which completes the proof that the Debye representation is injective.

    \section{$\bJ^-$ is a Physical Current}
    In this section we prove that the field $\bJ^-$ found by solving the
    London equation in $\Omega$ is, in fact, the physical current. By this we
    mean that the external field $\bB^+$ is found by applying the Biot-Savart
    operator to $\bJ^-.$ To establish this we need some notation, and further
    uniqueness results.

    We begin with a uniqueness theorem.
    \begin{theorem}\label{thm2}
Let $\Omega\subset\bbR^3$ be a bounded domain with smooth boundary. Suppose that
$\bB\in\cC^0(\bbR^3),$ satisfies
\begin{equation}
    \nabla \times \bB=0,\quad \nabla \cdot \bB=0\text{ in } \mathbb{R}^{3} \setminus \Gamma
    \end{equation}
and $\bB(\bx)=O(|\bx|^{-1}),$ as $x\to\infty,$ then $\bB\equiv 0.$
    \end{theorem}
    \begin{proof}
      The first observation is that $\bB$ must actually be a smooth harmonic
      vector field in all of $\bbR^3.$  A simple integration by parts argument shows that $\bB$ is a distributional
      solution to $\Delta \bB=0$, and so, by elliptic regularity it is a classical solution.
      Since $\nabla \times \bB=0$ in $\bbR^3$ it follows that
      \begin{equation}
        f(\bx)=\int_{0}^x \bB \cdot d\ell
      \end{equation}
      is a well defined function that satisfies
      \begin{equation}
        \nabla f=\bB, \quad \Delta f=0\text{ and }|f(\bx)|=O(|\log|\bx||)\text{ as }|\bx|\to\infty.
      \end{equation}
      The sharp form of Liouville's theorem implies that $f$ is constant, which
      completes the proof of the theorem.
    \end{proof}
    
    As a corollary we see that the solution to the following transmission
    problem, with given $\bJ^-$ in $\Omega$ and $\bB^{\In}$ on $\Gamma$ has a unique solution
    \begin{equation}\label{eqn45}
      \begin{split}
        \nabla \times \bB^-=\bJ^-,\quad \nabla \cdot \bB^-=&0\text{ in }\Omega \, ,\\
        \nabla \times \bB^+=0,\quad \nabla \cdot \bB^+=&0\text{ in }[\overline{\Omega}]^c\text{
          with } |\bB^+(\bx)|=O(|\bx|^{-1}) \, ,\\
        \bB^- - \bB^+ = \bB^{\In} \quad \text{on }\Gamma &
      \end{split}
    \end{equation}
    is unique as well.

    The second order of business is the Biot-Savart formula.
    \begin{theorem}[Biot-Savart Law] Suppose that $\bJ$ satisfies 
    $\nabla \cdot \bJ = 0$ in $\Omega$ and $\bJ \cdot \bn = 0$ 
    on $\Gamma$.  Let
    \begin{equation}
      \btheta(\bx)=\nabla \times \int_{\Omega} \frac{\bJ}{4\pi |\bx-\by|} dV \, ,
    \end{equation}
    then except on $\Gamma$,
    $$\nabla \times \btheta = \bJ$$
    \end{theorem}
\begin{proof}
   For any $\bx$ not in $\Gamma$, 
   \begin{equation}
   \begin{split}
   \nabla \times \btheta &= \nabla \times \nabla \times \int_{\Omega} \frac{\bJ}{4\pi |\bx-\by|} dV \\
   &= \nabla \nabla \cdot \int_{\Omega} \frac{\bJ}{4\pi |\bx-\by|} dV - \Delta \int_{\Omega} \frac{\bJ}{4\pi |\bx-\by|} dV\\
   &= \nabla \int_{\Omega}  \left(\nabla_{y} \cdot \left( \frac{\bJ}{4\pi | \bx-\by|} \right) - \frac{\nabla_{y} \cdot \bJ(\by)}{4\pi |\bx-\by|} \right) dV + \bJ \\
   &= \bJ + \nabla \int_{\Gamma} \frac{\bJ \cdot \bn}{4\pi |\bx-\by|} - \nabla \int_{\Omega} \frac{\nabla_{y} \cdot \bJ(\by)}{4\pi |\bx-\by |} dV = \bJ \,.
   \end{split}
   \end{equation}
      This establishes the
         Biot-Savart law, and makes explicit the hypotheses needed to make it
         hold.
\end{proof}

To complete this circle of ideas we find the formula  expressing a divergence 
free vector field $\bB$ in a domain $\Omega$  in terms of 
$\nabla \times \bB$ and appropriate boundary data.
\begin{theorem}\label{thm4}
  Let $\bB$ be a divergence free vector field defined in the smoothly bounded domain $\Omega.$ 
  Suppose that $\nabla \times \bB=\bJ,$ with $\nabla \times \bB \cdot \bn= \nabla_{\Gamma} \cdot \bn \times \bB = 0$ on $\Gamma$. 
  In this case
  \begin{equation}\label{eqn:Brep}
    \bB=\nabla \times \int_{\Omega}\frac{\bJ}{4\pi |\bx-\by|} dV + \bB_{t} + \bB_{n} \, ,
  \end{equation}
  where
  \begin{equation}
  \label{eq:bbdrydef}
    \bB_t=-\nabla \times \int_{\Gamma} \frac{\bn \times \bB}{4\pi |\bx-\by|} dS \quad \text{ and } \quad
    \bB_n= \nabla \int_{\Gamma} \frac{\bB \cdot \bn}{4\pi |\bx-\by|} dS
  \end{equation}
  are harmonic vector fields in $\Gamma^c.$
\end{theorem}
\begin{remark}
  Recall that a vector field $\bA$ is a harmonic vector field 
  if it satisfies the equations
  \begin{equation}
    \nabla \times \bA= \nabla \cdot \bA=0.
  \end{equation}
  This of course implies that $-\Delta\bA=(\nabla \times \nabla -\nabla \nabla \cdot)\bA=0,$
  and on a closed compact manifold these conditions are equivalent,
  but, in a domain with boundary, the former is a stronger condition.
  
  We let $\cH^{2}_{R}(\Omega^c)$ denote the vector space of outgoing harmonic vector fields in $\Omega^{c}$ that satisfy $\bB\cdot \bn = 0$ on $\Gamma$. For these vector fields~\cref{eqn:Brep} implies that
  \begin{equation}
  \bB(\bx) = -\nabla \times \int_{\Gamma} \frac{\bn \times \bB}{4\pi |\bx-\by|} dS \, \quad \text{for } \bx \in \Omega^{c} \, .
  \end{equation}
 It is a consequence of Hodge theory that $\cH^{2}_{R}(\Omega^{c}) \simeq H^2_{\dR}(\Omega^c,\Gamma)$, which is a vector space of dimension g, if $\Gamma$ has genus g. An element of this vector space is completely determined by the periods
 \begin{equation}
 \int_{A_{j}} \bB \cdot d \ell \, \quad \text{for } j=1,2\ldots,g.
 \end{equation}
 This shows that $\cH^{2}_{R}(\Omega^{c})$ is another $g-$dimensional vector space that could be used for $V_{+}$ in the definition of $\bA^{+}$.
\end{remark}
\begin{proof}
As before we begin with the
formula
\begin{equation}
\label{eq:lapsplit}
  -\Delta \int_{\Omega}\frac{\bB}{4\pi |\bx-\by|} dV = (\nabla \times \nabla  - \nabla \nabla \cdot) 
  \int_{\Omega}\frac{\bB}{4\pi |\bx-\by|} dV = \bB(\bx) \,.
\end{equation}
We need to compute the curl and the divergence of the integral. We see that
\begin{equation}
\label{eq:bcurlcomp}
\begin{split}
\nabla_{x} \times \int_{\Omega}\frac{\bB}{4\pi |\bx-\by|} dV &= 
\int_{\Omega} \frac{\nabla_{y} \times \bB}{4\pi |\bx-\by|}  dV - \int_{\Omega}
 \nabla_{y} \times \left(\frac{\bB}{4\pi |\bx-\by|}\right) dV  \\
 &= \int_{\Omega} \frac{\nabla_{y} \times \bB}{4\pi |\bx-\by|} dV - \int_{\Gamma} \frac{\bn \times \bB}{4\pi |\bx-\by|} dS \, ,
 \end{split}
\end{equation}
and
\begin{equation}
\label{eq:bdivcomp}
\begin{split}
\nabla_{x} \cdot \int_{\Omega}\frac{\bB}{4\pi |\bx-\by|} dV &= 
\int_{\Omega} \frac{\nabla_{y} \cdot \bB}{4\pi |\bx-\by|} dV - \int_{\Omega}
 \nabla_{y} \cdot \left(\frac{\bB}{4\pi |\bx-\by|}\right) dV  \\
 &= \int_{\Omega} \frac{\nabla_{y} \cdot \bB}{4\pi |\bx-\by|} dV - \int_{\Gamma} \frac{\bB \cdot \bn}{4\pi |\bx-\by|} dS \, .
 \end{split}
 \end{equation}
 The result then follows by combining~\cref{eq:lapsplit,eq:bcurlcomp,eq:bdivcomp}, and that $\bB$ is divergence free in $\Omega$. All that is left to show is that $\bB_{t}$ and $\bB_{n}$ are harmonic vector fields in $\Gamma^{c}$. 
 Trivially, $\nabla \cdot \bB_{t} = 0$, and $\nabla \times \bB_{n} = 0$. 
 Furthermore 
 \begin{equation}
 \nabla \cdot \bB_{n} = \Delta \int_{\Gamma} \frac{\bB \cdot \bn}{4\pi |\bx-\by|} dS = 0 \, ,
 \end{equation}
since $1/|\bx-\by|$ is a harmonic function for $x \in \Gamma^{c}$. All that remains to show is that
$\nabla \times \bB_{t} = 0$. To this end, observe that
\begin{equation}
\begin{split}
\nabla \times \bB_{t} &= -(\nabla \nabla \cdot  -\Delta) \int_{\Gamma} \frac{\bn \times \bB}{4\pi |\bx-\by|} dS \\
&= -\nabla \nabla \cdot \int_{\Gamma} \frac{\bn \times \bB}{4\pi |\bx-\by|} dS \, ,
\end{split}
\end{equation}
since $1/|\bx-\by|$ is a harmonic function. Thus, it suffices to show that
\begin{equation}
\nabla \cdot \int_{\Gamma} \frac{\bn \times \bB}{4\pi |\bx-\by|} dS = 0 \, .
\end{equation}
To this end observe that
\begin{equation}
\label{eq:brep2}
\nabla \cdot \int_{\Gamma} \frac{\bn \times \bB}{4\pi |\bx-\by|} dS = 
\int_{\Gamma} \frac{\nabla_{\Gamma} \cdot \bn \times \bB}{4 \pi | \bx-\by|} dS = 0
\end{equation}
\end{proof}

Observe that the first term in~\cref{eq:brep2} is precisely the Biot-Savart law, but note that the solution to 
$\nabla \times \bB = \bJ$, $\nabla \cdot \bB = 0$ in $\Omega$ with $\bJ \cdot \bn =0$ on $\Gamma$ is not generally
given by the Biot-Savart formula alone, but may require the addition of a harmonic vector field determined by the boundary
values of $\bB$. 

This finally brings us to the main result of this section.
\begin{theorem}
  Suppose that $\bB\in \cC^0(\bbR^3)$ solves~\eqref{eqn45}. If
  $\bB^{\In}$ extends to $\Omega$ as a solution to $\nabla \times \bB^{\In}=0,
  \nabla \cdot \bB^{\In}=0,$ then, for $\bx\in\Omega^c$ we have that
  \begin{equation}
    \bB^+(\bx)=\nabla \times \int_{\Omega}\frac{\bJ^-(\by)}{4\pi |\bx-\by|} dV
  \end{equation}
  That is, the outgoing field is generated by the currents within the superconductor.
\end{theorem}
\begin{proof}
  The proof is based on~\cref{thm2}, the foregoing computations, and
  the jump formul{\ae} for the boundary integrals in~\cref{eq:bbdrydef}, see~\cite{ColtonKress}. Because $\bB^+$
  is outgoing, we observe that~\cref{eqn:Brep} can be applied to represent
  $\bB^+$ for $\bx\in\Omega^c:$
  \begin{equation}\label{eqn63}
    \bB^+(\bx)=\nabla \times \int_{\Gamma} \frac{\bn \times \bB^+}{4\pi |\bx-\by|} dS -
    \nabla \int_{\Gamma} \frac{\bB^+ \cdot \bn}{4\pi |\bx-\by|} dS \, .
  \end{equation}
  Here the signs are reversed because we are representing a field in
  $\Omega^c,$ and we orient $\Gamma$ as the boundary of $\Omega.$  The
  formula in~\eqref{eqn63} also defines a harmonic field in $\Omega.$ It follows from the
  jump relations for these integrals that this field has zero boundary data and
  is therefore zero in $\Omega.$  We can similarly represent $\bB^{\In}(\bx)$ for
  $\bx\in\Omega$ in this form:
  \begin{equation}\label{eqn64}
     \bB^{\In}(\bx)=-\nabla \times \int_{\Gamma} \frac{\bn \times \bB^{\In}}{4\pi |\bx-\by|} dS +
    \nabla \int_{\Gamma} \frac{\bB^{\In} \cdot \bn}{4\pi |\bx-\by|} dS \, .
  \end{equation}

  Finally we represent $\bB^-$ for $x\in\Omega$ as
   \begin{equation}\label{eqn65}
    \bB^-(\bx)=\nabla \times \int_{\Omega} \frac{\bJ}{4\pi |\bx -\by|} dV -\nabla \times \int_{\Gamma} \frac{\bn \times \bB^-}{4\pi |\bx-\by|} dS +
    \nabla \int_{\Gamma} \frac{\bB^- \cdot \bn}{4\pi |\bx-\by|} dS
      \end{equation}
   Using the boundary condition, $\bB^-=\bB^++\bB^{\In}$ along
   $\Gamma,$ and the fact, noted above, that the $\bB^+$-contribution will
   be zero for $x\in\Omega,$ gives
   \begin{equation}\label{eqn66}
   \bB^-(\bx)=\nabla \times \int_{\Omega} \frac{\bJ}{4\pi |\bx -\by|} dV -\nabla \times \int_{\Gamma} \frac{\bn \times \bB^{\In}}{4\pi |\bx-\by|} dS +
    \nabla \int_{\Gamma} \frac{\bB^{\In} \cdot \bn}{4\pi |\bx-\by|} dS \, .
   \end{equation}
   If we define
   \begin{equation}
     \tbB^+(\bx)=\nabla \times \int_{\Omega}\frac{\bJ^-}{4\pi |\bx-\by|} dV\text{ for }\bx\in\Omega^c,
   \end{equation}
   then $\tbB^+$ is harmonic and outgoing in $\Omega^c.$ The fact that the Biot-Savart
   term is continuous across $\Gamma$ along with equations~\eqref{eqn64}
   and~\eqref{eqn66} imply that, for $\bx_0\in\Gamma,$ we have:
   \begin{equation}
     \begin{split}
       \lim_{x\to x_0^-}\bB^-(\bx)-&\lim_{x\to x_0^+}\tbB^+(\bx)\\
       &=
    \lim_{x\to x_0^-}\left[ -\nabla \times \int_{\Gamma} \frac{\bn \times \bB^{\In}}{4\pi |\bx-\by|} dS +
    \nabla \int_{\Gamma} \frac{\bB^{\In} \cdot \bn}{4\pi |\bx-\by|} dS
    \right]\\
      &=\bB^{\In}(\bx_0).
      \end{split}
   \end{equation}
   The conclusion of the theorem now follows from the corollary of~\cref{thm2}.
\end{proof}

As a corollary of the proof we see that if $\bB^{\In}\equiv 0,$ then we have
the representation formula
\begin{equation}\label{eqn69}
  \bB(\bx)=\nabla \times \int_{\Omega}\frac{\bJ^-(\by)}{4\pi |\bx-\by|}dV
\end{equation}
holding globally throughout $\bbR^3.$ This fact can also be deduced directly
from Theorem~\ref{thm2}, since the field defined in~\eqref{eqn69} is easily seen
to be a solution to~\eqref{eqn45} with $\bB^{\In}=0.$

\section{Numerical results}
\label{sec:numres}
The system of Fredholm equations for the unknown functions $q^{-},r^{-},q^{+}$ and the
$4g$ constants which define $\bl_{H}^{\pm}$ is given by equations~\eqref{eqn26},~\eqref{eqn27},~\eqref{eqn29}, and~\eqref{eqn30}. However, in order to compute to $\bl^{-}$ and $\bm^{-}$, we require access to $\Delta_{\Gamma}^{-1}$ which is not analytically known on complicated surfaces.
Let $f^{-}$, $g^{-}$, and $f^{+}$ be the solutions to the following Laplace-Beltrami problems,
\begin{equation}
\left( \Delta_{\Gamma} + \cW \right) f^{-}= q^{-} \, , \quad (\Delta_{\Gamma} + \cW) f^{+} = q^{+} - \cW[q^{+}] \, ,\quad \text{and} \, \quad (\Delta_{\Gamma} + \cW) g^{-} = r^{-} \, ,
\end{equation}
where $\cW[q]$ is the constant function whose magnitude is the average value of $q$ on the surface. 
It can be shown that $\Delta_{\Gamma} + \cW$, is an invertible operator, which maps  the space of mean zero functions to itself.
We solve Laplace-Beltrami problems by representing the solution of $\Delta_{\Gamma} u = f$, as $u = \cS_{0}^{2}[\sigma]$. Suppose that $f^{-} = \cS_{0}^2[\rho^{-}]$, $f^{+} = \cS_{0}^{2}[\rho^{+}]$, and $g^{-} = \cS_{0}^2[\mu^{-}]$.
Then it follows from~\cref{eq:identites} that $\rho^{+}$, $\rho^{-}$, and $\mu^{-}$ satisfy
\begin{equation}
\label{eqn:lapbelinv}
\begin{aligned}
-\frac{\rho^{-}}{4} + \left( (\cS_{0}')^2 - \left( \cS_{0}'' + \cD_{0}' - 2H \cS_{0}' \right)  \cS_{0} + \cW \right)[\rho^{-}] &= q^{-} \, , \\
-\frac{\rho^{+}}{4} + \left( (\cS_{0}')^2 - \left( \cS_{0}'' + \cD_{0}' - 2H \cS_{0}' \right)  \cS_{0} + \cW \right)[\rho^{-}] &= q^{+} - \cW[q^{+}]\, , \\
-\frac{\mu^{-}}{4} + \left( (\cS_{0}')^2 - \left( \cS_{0}'' + \cD_{0}' - 2H \cS_{0}' \right)  \cS_{0} + \cW \right)[\rho^{-}] &= r^{-} \, .
\end{aligned}
\end{equation}
The surface currents $\bl^{-}$, and $\bm^{-}$ can be expressed in terms of $\rho^{-},\rho^{+},$ and $\mu^{-}$ as
\begin{equation}
\begin{split}
  \bm^- &=\frac{1}{\lambda_L}\nabla_{\Gamma} \cS_{0}^2 [\rho^-]\\
  \bl^-&=-\frac{1}{\lambda_L}\left[\nabla_{\Gamma} \cS_{0}^2[\mu^-] + \bl_H^-  \right] \, .
  \end{split}
\end{equation}
We then solve the Fredholm system of equations~\eqref{eqn26},~\eqref{eqn27},~\eqref{eqn29},~\eqref{eqn30}, and~\eqref{eqn:lapbelinv} for the functions $q^{-},q^{+},r^{-},\rho^{-},\rho^{+},\mu^{-}$, and the $4g$ constants defining $\bl_{H}^{\pm}$. 
All of the layer potentials that arise in these equations are numerically computed using the fast multipole accelerated locally corrected quadrature method of~\cite{fmm3dbiepaper} implemented in the software packages~\texttt{fmm3dbie} (\url{https://gitlab.com/fastalgorithms/fmm3dbie}) and accelerated using the fast multipole library~\texttt{FMM3D} (\url{https://github.com/flatironinstitute/FMM3D}).
The harmonic vector fields on the surface are also not analytically known on general geometries. These are numerically computed using the procedure discussed in~\cite{lapbel1,lapbel2}.

\subsection{Accuracy \label{sec:acc}}
We demonstrate the accuracy and convergence of our approach by constructing an incident field for which the interior and exterior magnetic fields and interior currents can be computed.  Let $\bx_{o}$ denote a point in the exterior region $\Omega^{c}$ and $\bx_{i}$ denote a point in the interior $\Omega$. 
Let $\bJ^{-}_{0}$ and $\tbB^{-}_{0}$ denote the fields constructed due to a point source located at $\bx_{o}$ and let $\bB^{+}_{0}$ be the field due to a point charge at $\bx_{i}$, i.e.
\begin{equation}
\begin{aligned}
\bB^{+}_{0} = \nabla g_{0}(\bx,\bx_{i}) \, , \quad \, \bJ^{-}_{0} = \text{Re} \left(-i \nabla \times \nabla \times \left(\bv_{o} g_{k}(\bx,\bx_{o}) \right) \right) \, , \quad \tbB^{-}_{0} = \text{Re} \left(k \nabla \times \left( \bv_{o} g_{k}(\bx,\bx_{o}) \right)\right) \, ,
\end{aligned}
\end{equation}
where $k = i/\lambda_{L}$, and $\bv_{o} \in \mathbb{C}^{3}$ is a constant.

Suppose that we solve the system of equations with $\bB^{\In} = \lambda_{L} \tbB^{-}_{0} - \bB^{+}_{0}$ and the boundary data for equation~\eqref{eqn28} is set to $\bJ^{-}_{0} \cdot \bn$ instead of $0$. Using a proof similar to the uniqueness result proved in~\cref{sec:uniqueness}, the solution to the system of equations yields
$(\tbB^{-}, \bJ^{-}) = (\tbB^{-}_{0}, \bJ^{-}_{0})$, for $\bx \in \Omega$ and $\bB^{+} = \bB^{+}_{0}$ for $\bx \in \Omega^c$. 
Given the solution of the system of integral equations, we then compute the fields $\tbB^{-},\bJ^{-}$ at $10$ points in $\Omega$ denoted $\{\bt_{j}^{-}:\: j=1,\dots,10\}$, the field $\bB^{+}$ at $10$ points in $\Omega^{c},$ denoted  $\{\bt_{j}^{+}:\: j=1,\dots,10\}$. 
Let $\tbB^{-}_{c}, \bJ^{-}_{c}$, and $\bB^{+}_{c}$ denote the numerically computed magnetic fields and current. 
Let $\varepsilon_{1}$ denote the relative $L^2$ error in the computed fields at the target locations given by
\begin{equation}
\varepsilon_{1} = \frac{1}{M}\left[\sum_{j=1}^{10} \left|\tbB^{-}_{c} (\bt^{-}_{j}) - \tbB^{-}(\bt^{-}_{j})\right|^2 + \left|\bJ^{-}_{c} (\bt^{-}_{j}) - \bJ^{-}(\bt^{-}_{j})\right|^2 + \left|\bB^{+}_{c} (\bt^{+}_{j}) - \bB^{+}(\bt^{+}_{j})\right|^2\right]^{\frac 12},
\end{equation}
where
\begin{equation}
    M=\sqrt{\int\limits_{\Gamma} \left(|q^{+}|^2 + |q^{-}|^2 + |r^{-}|^2\right) \, dS}.
\end{equation}
Similarly, let $\varepsilon_{2}$ denote the relative $L^{2}$ error for computing the fields on the boundary $\Gamma$ given by
\begin{equation}
\varepsilon_{2} = \frac{1}{M}\left[\bigintss\limits_{\Gamma} \left( \left|\tbB^{-}_{c} - \tbB^{-}\right|^2 + \left|\bJ^{-}_{c} - \bJ^{-}\right|^2 + \left|\bB^{+}_{c}  - \bB^{+}\right|^2 \right) \, dS\right]^{\frac 12}
\end{equation}
Note that this is not an inverse crime, since we do not impose boundary conditions on the tangential components of $\tbB^{-},\bJ^{-}$, and $\bB^{+}$. Furthermore, to measure the relative $L^{2}$ error, we scale by the $L^{2}$ norm, $M,$ of the densities $q^{+},q^{-},r^{-}$,  since the quadrature error estimates for evaluating layer potentials using fmm3dbie guarantees accuracy relative to the norm of the density~\cite{fmm3dbiepaper}.

For our numerical experiments, we consider the unit sphere, a twisted torus geometry, and a surface of genus 2. The boundary $\Gamma$ for the twisted torus geometry is parameterized by $\bX: [0,2\pi]^2 \to \Gamma$
with
\begin{equation}
    \bX(u,v) = \sum_{i=-1}^{2} \sum_{j=-1}^{1} \delta_{i,j} \begin{bmatrix} \cos{v} \cos{((1-i)\, u+j\, v)} \\
    \sin{v} \cos{((1-i) \, u+j \, v)} \\
    \sin{((1-i)\, u+j \, v)} 
    \end{bmatrix} \, ,
\end{equation}
where the non-zero coefficients are $\delta_{-1,-1}=0.17$, 
$\delta_{-1,0} = 0.11$, $\delta_{0,0}=1$, $\delta_{1,0}=4.5$, $\delta_{2,0}=-0.25$, $\delta_{0,1} = 0.07$, and $\delta_{2,1}= -0.45$.
(See Fig. \ref{fig:1}.) For the remainder of the paper, we will refer to this geometry as the {\em stellarator}.

Let $\Npat$ denote the number of patches describing the surface $\Gamma$,  and $N$ the total number of discretization points (the system size then would be $m = 6N +4g$). Each patch is discretized using $9$th order Vioreanu-Rokhlin discretization nodes~\cite{spectra}. In~\cref{tab:acc1}, we tabulate the errors $\varepsilon_{1}$, $\varepsilon_{2}$, and $\niter$, the number of GMRES iterations required for the relative residual to be less than $10^{-8}$ for the solutions on the unit sphere, and a stellarator, with $\lambda_{L}=1$. 

The tolerance for evaluating the layer potentials and the fast multipole acceleration $\varepsilon$ was set to $5 \times 10^{-7}$.
For the sphere, we observe the expected convergence rate of $O(h^9 + \varepsilon)$, while for the stellarator we observe a convergence rate of $O(h^8 + \varepsilon)$. The reduction in order of convergence for the stellarator is in part explained by the loss of an order of convergence in computing the harmonic vector fields as noted in~\cite{lapbel1,lapbel2}.

\begin{table}[h!]
\begin{center}
\begin{tabular}{cccccSS} \hline
     {Geometry} & $\lambda_{L}$ & {$\Npat$} & {$N$} & {$\niter$} & {$\varepsilon_{1}$  } & {$\varepsilon_{2}$}   \\ \hline
    \multirow{3}{*}{Sphere} & 1 & 12  & 540 & 23 & 5.0E-3 & 1.4E-2  \\    
    & 1& 48 & 2160 & 17 & 1.2E-6 & 2.6E-5  \\
    & 1 & 192 & 8640 & 17 & 2.6E-8 & 1.8E-7  \\ \hline
    
        \multirow{3}{*}{Stellarator} & 1 & 150 & 6750 & 78   & 1.4E-3 & 1.7E-2  \\    
    & 1 &t 600 & 27000 &  86 & 4.6E-5 & 8.1E-5  \\
    & 1 & 2400 & 108000 & 70 & 3.2E-8 & 5.7E-7   \\ \hline
\end{tabular}
\caption{Relative $L^{2}$ error in computing known analytic solution computed using the generalized Debye formulation. $\Npat$ is the number of patches, $N$ the number of discretization nodes, $\niter$ the number of GMRES iterations for the relative residual to be less than $10^{-8}$, $\varepsilon_{1}$ is the relative $L^2$ error in the evaluated solution at targets off-surface, and $\varepsilon_{2}$ is the relative $L^{2}$ error in the fields on surface.} \label{tab:acc1}
   \end{center}
\end{table}

In~\cref{tab:acc2}, we tabulate the errors $\varepsilon_{1}$, $\varepsilon_{2}$, and $\niter$ on the stellarator for $\lambda_{L} = 1,10^{-1}$, and $10^{-1}/3$. As we show in a subsequent paper,  see~\cite{EpRa2}, the second kind Fredholm system approaches a mixed first kind-second kind Fredholm system of equations as $\lambda_{L} \to 0$. Thus, the increase in number of GMRES iterations required for the solution to converge to the same residual is expected. The loss of accuracy for $\varepsilon_{2}$ as $\lambda_{L}$ is increased can be in part attributed to the quadrature error owing to the more singular nature of the kernels of the layer potentials. However, with a sufficiently fine mesh, for a fixed $\lambda_{L}$, we observe the expected convergence behavior of $O(h^8 + \varepsilon)$.

\begin{table}[h!]
\begin{center}
\begin{tabular}{cccccSS} \hline
{Geometry} & $\lambda_{L}$ & {$\Npat$} & {$N$} & {$\niter$} & {$\varepsilon_{1}$  } & {$\varepsilon_{2}$}   \\ \hline
           \multirow{4}{*}{Stellarator} & 1 & 600 & 27000 & 86   & 4.6E-5 & 8.1E-5  \\    
    & $10^{-1}$ & 600 & 27000 &  168 & 4.7E-5 & 4.9E-4  \\
    & $10^{-1}/3$ & 600 & 27000 & 230 & 5.2E-5 & 1.7E-2   \\ 
    & $10^{-1}/3$ & 2400 & 108000 & 191 & 2.9E-9 & 1.2E-5   \\ \hline
\end{tabular}
\caption{Relative $L^{2}$ error in computing known analytic solution computed using the generalized Debye formulation as $\lambda_{L}$ is varied. $\Npat$ is the number of patches, $N$ the number of discretization nodes, $\niter$ the number of GMRES iterations for the relative residual to be less than $10^{-8}$, $\varepsilon_{1}$ is the relative $L^2$ error in the evaluated solution at targets off-surface, and $\varepsilon_{2}$ is the relative $L^{2}$ error in the fields on surface.} \label{tab:acc2}
   \end{center}
\end{table}

\begin{figure}[h!]
    \centering
    \includegraphics[width=0.7\linewidth]{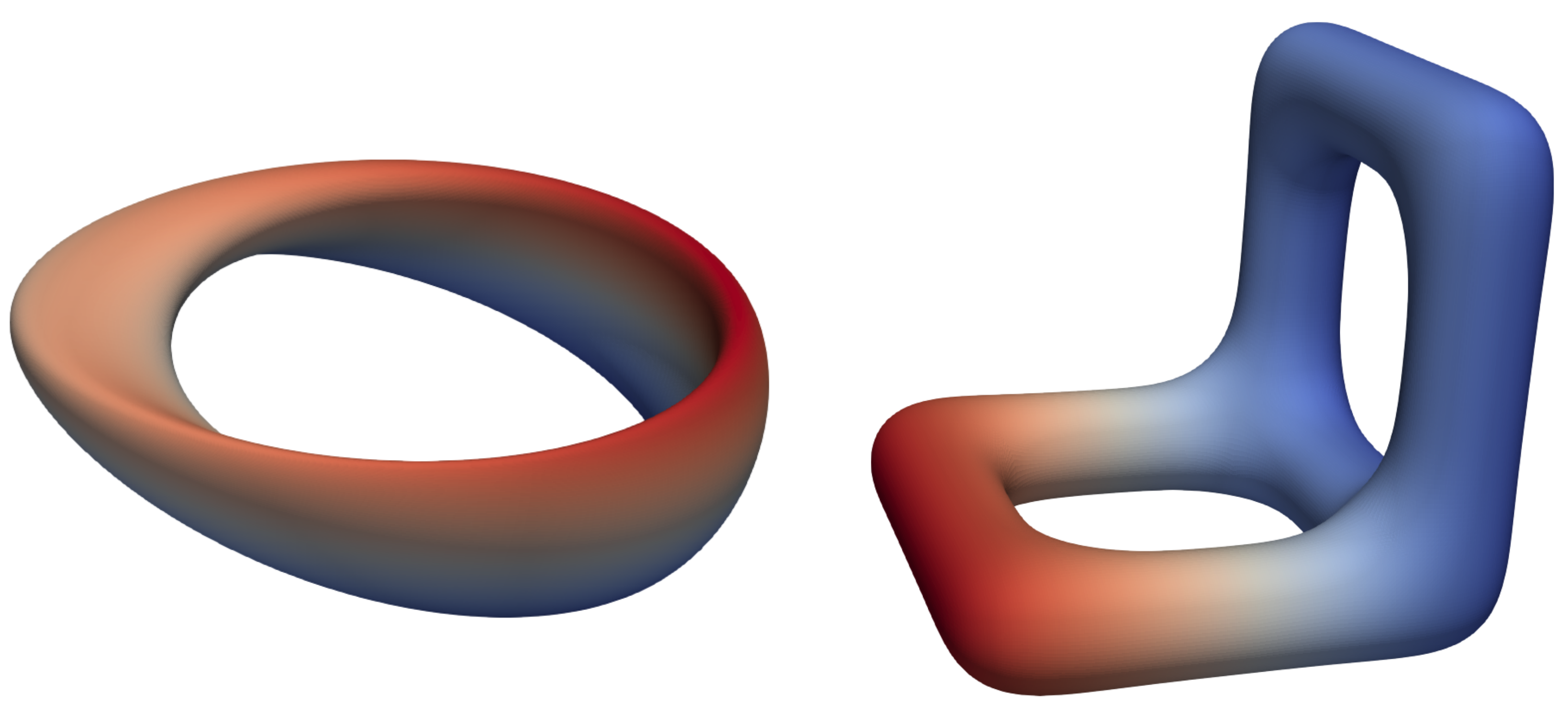}
    \caption{The boundary of a stellarator like geometry (left) and a genus 2 torus (right). The stellarator is colored using the z-coordinate.}
    \label{fig:1}
\end{figure}

\subsection{Condition number estimation on the sphere}
\label{subsec:sphere}
As noted in~\cref{rem:rep}, there exists a family of Debye representations for $\bl^{-}$ and $\bm^{-}$ which result in invertible second kind systems of integral equations for the solution of static currents in \mbox{type-I} superconductors. In this section, we study how the condition number of the resultant Fredholm system is impacted by the choice of these parameters. On the unit sphere, one can analytically compute the action of the system of integral equations on the spherical harmonics.  Moreover, owing to symmetry of the problem defined on the sphere, the action of the integral operator only depends on degree of the spherical harmonic and is independent of the order.  Detailed analyses for the spectra of various layer potentials on the sphere are discussed in~\cite{boundaryanalysis} for example. 

 We wish to explore the effect of $\sigma_{m}, \sigma_{\ell}$ on the conditioning of the system where
\begin{equation}
 \begin{split}
  \bm^- &=\frac{1}{\lambda_L}\left[ \nabla_{\Gamma} \Delta_{\Gamma}^{-1}q^-
  +\sigma_{m} \bn \times \nabla_{\Gamma} \Delta_{\Gamma}^{-1} \left(q^{+} - \frac{1}{|\Gamma|} \int_{\Gamma} q^{+} \right)  \right]\\
  \bl^-&=-\frac{1}{\lambda_L}\left[\nabla_{\Gamma} \Delta_{\Gamma}^{-1}r^- 
  + \sigma_{\ell} \bn \times \nabla_{\Gamma} \Delta_{\Gamma}^{-1} \left(q^{+} - \frac{1}{|\Gamma|} \int_{\Gamma} q^{+} \right)  
  \right] \, , 
  \end{split}
\end{equation}
Suppose that the linear system is written as map from
map from $[\rho^{-}, \rho^{+}, \mu^{-}, q^{-}, q^{+}, r^{-}]$ 
to the boundary conditions ordered in the following manner on the sphere
\begin{equation}
\begin{aligned}
 (\Delta_{\pa \Omega} + \cW) \cS_{0}^2 [\rho^{-}] -  q^{-} &= 0 \\
 (\Delta_{\pa \Omega} + \cW) \cS_{0}^2 [\rho^{+}] -  q^{+} &= 0 \\
 (\Delta_{\pa \Omega} + \cW) \cS_{0}^2 [\mu^{-}] -  r^{-} &= 0 \\
-\cS_{0}\left [\nabla_{\Gamma} \cdot \bn \times \bn \times \left(\lambda_{L} \tbB^{-} -  \bB^{+} \right) \right] &= -\cS_{0}[\nabla_{\Gamma} \cdot \bn \times \bn \times \bB^{\In}] \\
\left(\lambda_{L} \tbB^{-} - \bB^{+} \right) \cdot \bn &= \bB^{\In} \cdot \bn \\
\bJ^{-} \cdot \bn &=\bJ^{\In} \cdot \bn \, ,
\end{aligned}
\end{equation}

Suppose that $Y_{nm}(\theta,\phi)$ is the spherical harmonic of degree $n$ and order $m$. Let $j_{n}(r), h_{n}(r)$ denote the spherical Bessel function and spherical Hankel function of degree $n$ respectively and let $\delta_{0}(n)$ denote the Kronecker-delta function which is $1$ for $n=0$ and $0$ otherwise.
Let $\cA$ denote this linear operator. If $[\rho^{-}, \rho^{+}, \mu^{-}, q^{-}, q^{+}, r^{-}] = \bc Y_{nm}(\theta, \phi)$ where $\bc \in \mathbb{C}^{6}$ is a constant vector, then 
\begin{equation}
\cA \begin{bmatrix} 
\rho^{-} \\
\rho^{+} \\
\mu^{-} \\
q^{-} \\
q^{+} \\
r^{-}
\end{bmatrix} = \left(\cA_{n} \bc \right) Y_{nm} (\theta,\phi) \, ,
\end{equation}
where $\cA_{n}$ is the $6 \times 6$ matrix given by
\begin{equation}
\cA_{n} = 
\begin{bmatrix}
\frac{-n(n+1) + \delta_{0}(n)}{(2n+1)^2} & 0 & 0 & -1 & 0 & 0 \\
0 & \frac{-n(n+1) + \delta_{0}(n)}{(2n+1)^2} & 0 & 0 & -1 & 0 \\
0 & 0 & \frac{-n(n+1) + \delta_{0}(n)}{(2n+1)^2} & 0 & 0 & -1  \\
a & \frac{\sigma_{\ell} n (n+1) \left( j_{n}(k) + k j_{n}'(k) \right) h_{n}(k) }{(2n+1)^3 \lambda_{L}} & 0 & \frac{n(n+1) \left(j_{n}(k) h_{n}(k) \right)}{2n+1} & \frac{n(n+1)}{(2n+1)^2} & 0 \\
-\frac{b}{(2n+1)} & -\frac{\sigma_{\ell} n(n+1) j_{n}(k) h_{n}(k)}{\lambda_{L}(2n+1)^2} & 0 & -\frac{i j_{n}'(k)h_{n}(k)}{\lambda_{L}} & \frac{(n+1)}{2n+1} & 0 \\
0 & -\frac{\sigma_{m} j_{n}(k)h_{n}(k)n(n+1)}{\lambda_{L}^2 (2n+1)^2} & \frac{b}{\lambda_{L}} & 0 & 0 & \frac{i j_{n}'(k)h_{n}(k)}{\lambda_{L}^2} 
\end{bmatrix} \, ,
\end{equation}
where
\begin{equation}
\begin{aligned}
a &= \frac{n(n+1) \left( i\left((n+1)j_{n}(k)h_{n-1}(k) + nj_{n+1}(k)h_{n}(k)-kj_{n+1}(k)h_{n-1}(k) \right) \right)}{\lambda_{L} (2n+1)^3} \, , \\
b &= \frac{in(n+1) \left(j_{n}(k)h_{n-1}(k)-j_{n+1}(k)h_{n}(k) \right)}{\lambda_{L}(2n+1)} \, ,
\end{aligned}
\end{equation}
and $k= i/\lambda_{L}$.
For the Fredholm system above, the ``identity" block for the system of equations is given by
\begin{equation}
D = \begin{bmatrix}
-1/4 & 0 & 0 & 0 & 0 & 0 \\
0 & -1/4 & 0 & 0 & 0 & 0 \\
0 & 0 & -1/4 & 0 & 0 & 0 \\
0 & 0 & 0 & -\lambda_{L}/4 & 1/4 & 0 \\
0 & 0 & 0 & \lambda_{L}/2 & 1/2 & 0 \\
0 & 0 & 0 & 0 & 0 & -1/2 
\end{bmatrix} \, .
\end{equation}

In practice, to improve the conditioning of system of Fredholm equations, typically the system is preconditioned with $D^{-1}$ which is done for the following as well. Let $s_{j,n}$, $j=1,2,\ldots 6$ denote the singular values of $D^{-1} \cA_{n}$. 
Let 
\begin{equation} 
\smax = \max_{\substack{j=1,2,\ldots 6\\n \geq 0}} s_{j,n} \, , \quad \text{and} \quad \min_{\substack{j=1,2,\ldots 6\\n \geq 0}} s_{j,n}
\end{equation}
The condition number of the system of equations denoted by $\kappa$ is then given by $\kappa = \smax/\smin$. In~\cref{fig:2}, we plot $\kappa$ as a function of $\sigma_{m}$ and $\sigma_{\ell}$ for $\lambda_{L}=1,10^{-1}, 10^{-2}$. As we can see from the figures, for all values of $\lambda_{L},$ the system seems to be optimally conditioned when $\sigma_{m}=\sigma_{\ell}=0.$

\begin{figure}[h!]
    \centering
    \includegraphics[width=0.9\linewidth]{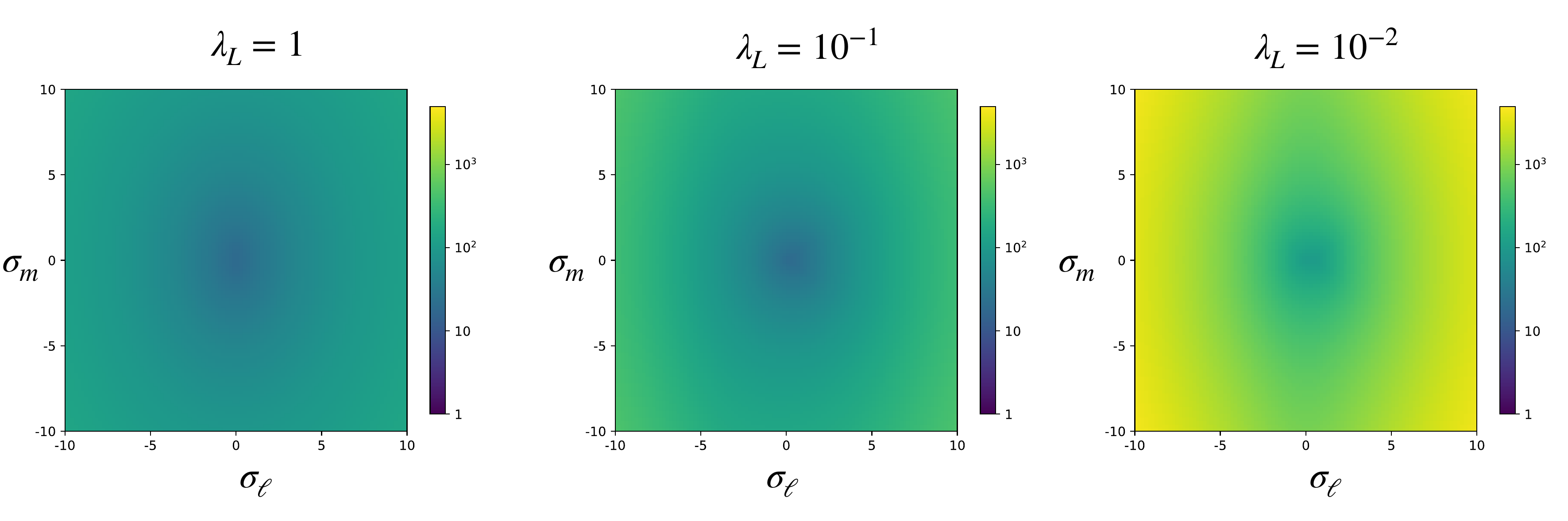}
    \caption{$\kappa$ as a function of $\sigma_{\ell},\sigma_{m}$ for $\lambda_{L}=1$ (left) ,$\lambda_{L}=10^{-1}$ center, and $\lambda_{L}=10^{-2}$ (right).}
    \label{fig:2}
\end{figure}

\begin{remark}
Since the system of equations is a second kind Fredholm equation, one can show that $D^{-1} \cA_{n} \to I$ as $n\to \infty$. Thus in practice, one can truncate the computation of $\smax$ and $\smin$ to $n\leq \nmax (\lambda_{L})$.
\end{remark}
	
\subsection{Static currents}
We now turn our attention to the computation of the static currents in the stellarator and a genus $2$ torus (see~\cref{fig:1}).
In~\cref{fig:stell-statj}, we plot the surface current $\bJ^{-}$ on the stellarator with $\lambda_{L} =1,10^{-1}$, and $10^{-1}/3$. For $\lambda_{L} = 1,10^{-1}$, we discretize the stellarator with $\Npat=600$ $8$th order patches, and for $\lambda_{L} = 10^{-1}/3$, we discretize the stellarator with $\Npat=2400$ $8$th order patches. 
In~\cref{fig:stell-statj-cross-sec}, we plot the current $\bJ^{-} \cdot \bn$ inside the cross section of the stellarator on the XZ plane at $y=0$. Here $\bn=(0,1,0).$ The current decays in the interior as $O(e^{-d(x,\pa \Omega)/\lambda_{L}})$ as expected.

\begin{figure}[h!]
    \centering
    \includegraphics[width=\linewidth]{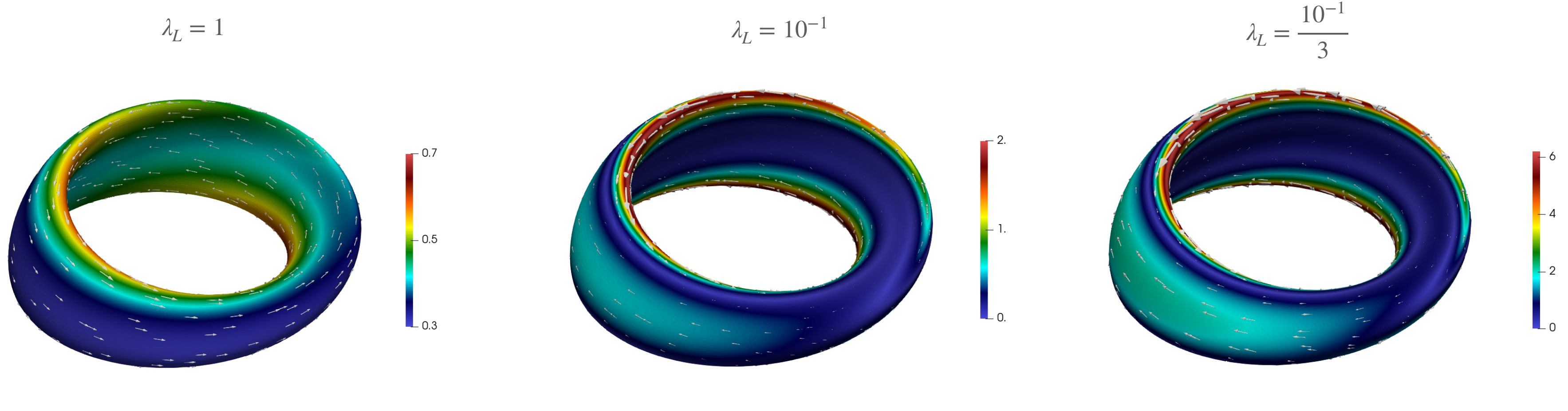}
    \caption{Static currents $\bJ^{-}$ on the surface of the stellarator for $\lambda_{L} = 1$ (left), $\lambda_{L} = 10^{-1}$ (center), and $\lambda_{L} = 10^{-1}/3$ (right).}
    \label{fig:stell-statj}
\end{figure}

\begin{figure}[h!]
    \centering
    \includegraphics[width=\linewidth]{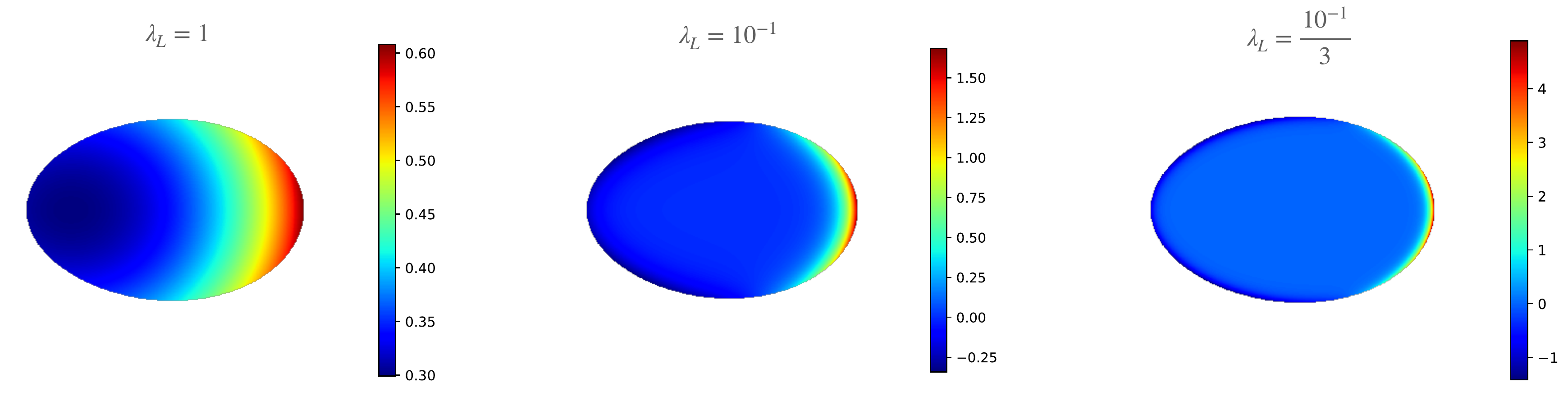}
    \caption{Static currents $\bJ^{-} \cdot \bn$ inside the stellarator on the XZ plane $y=0$ for $\lambda_{L} = 1$ (left), $\lambda_{L} = 10^{-1}$ (center), and $\lambda_{L} = 10^{-1}/3$ (right).}
    \label{fig:stell-statj-cross-sec}
\end{figure}

In~\cref{fig:genus2-statj}, we plot a static current, and the corresponding magnetic fields in the genus $2$ torus with $\lambda_{L}=1$. Both the current fluxes $a_{j}$, $j=1,2,$ were set to $1$. The torus is contained in a bounding box of size $4\times 4 \times 4$. A high order representation of the geometry was obtained by using the surface smoothing algorithm of~\cite{surface-smoother} on a second order triangulation of a genus 2 domain formed by merging two hollow rectangular parallelepipeds. The geometry was discretized with $\Npat=2496$, $9th$ order triangles. The errors $\varepsilon_{1},\varepsilon_{2}$ for the analytic solution test discussed in~\cref{sec:acc} are $\varepsilon_{1} = 1.4 \times 10^{-5}$, and $\varepsilon_{2} = 2.5 \times 10^{-4}$.

\begin{figure}[h!]
    \centering
    \includegraphics[width=0.8\linewidth]{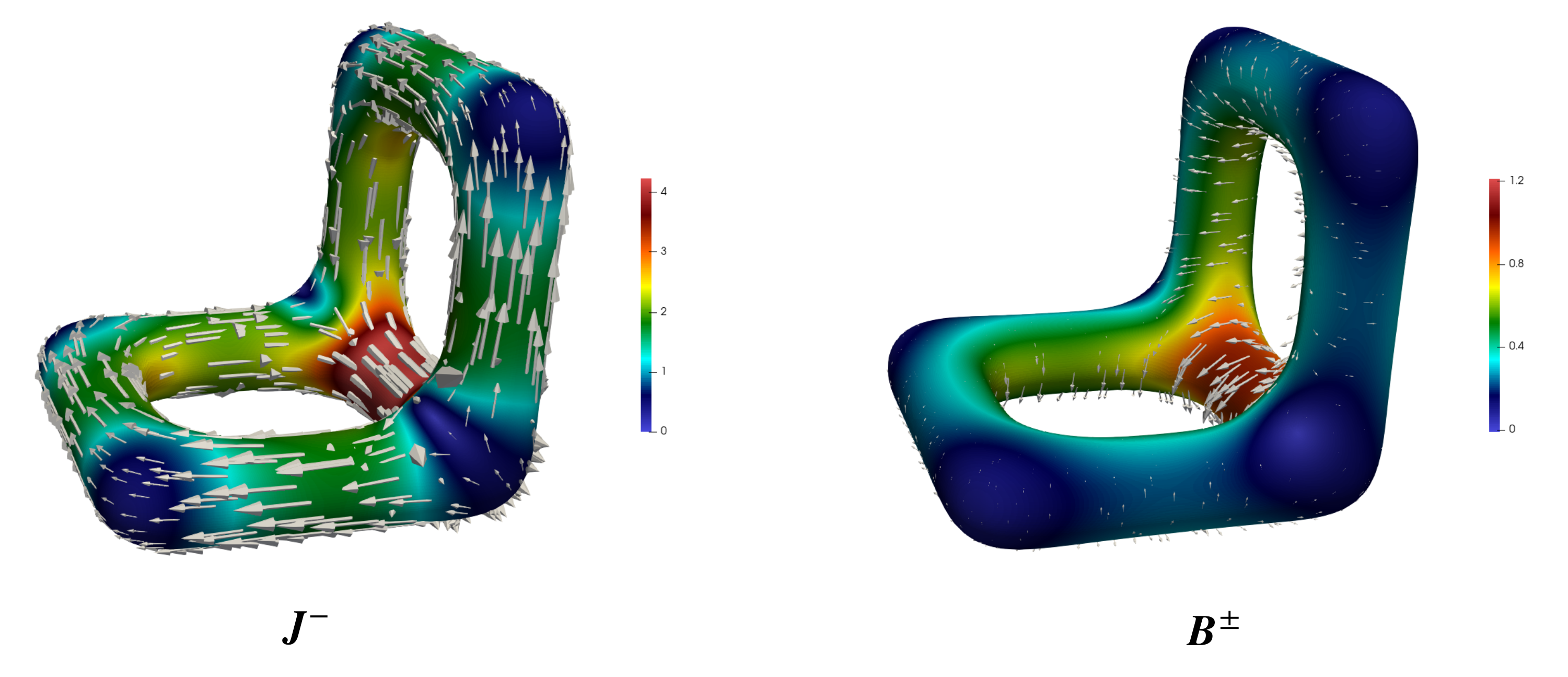}
    \caption{Static current $\bJ^{-}$, and the magnetic fields $\bB^{\pm}$ on the surface of a genus 2 torus for $\lambda_{L} = 1$. }
    \label{fig:genus2-statj}
\end{figure}

%Note that this implies an additional equation satisfied by $\bj^-$ on the
%$\Gamma:$
%\begin{equation}
%  \star
%  d\int_{\Gamma}\frac{d^*\bj^{-}}{4\pi|\bx-\by|}-d\star\int_{\Gamma}
%  d\star\bj^{-}(\by)\wedge\bomega_2(x,y)=0\text{ for all }x\in\Omega.
%\end{equation}
%Could we use this formula, perhaps taking $x\to\Gamma,$ along with a Debye
%source representation for $\bj^-,$ as a solution to
%$d^*d\bj^-=-\frac{1}{\lambda_L^2}\bj^-, d^*\bj^-=0,$ to find $\bj^-$ directly,
%without mention of $\bEta^-?$ 

\section{Conclusions}
In this note we have studied the standard scattering problem for a static, low-field, type-I superconductor using the standard London equation formulation within the superconducting material. Amongst other things, we have shown that, for a superconductor of genus $g>0,$ the fluxes through the $A$-cycles can be specified arbitrarily and uniquely determine a current distribution and a corresponding external, outgoing magnetic field. Moreover this field can be obtained by applying the Biot-Savart formula to the current within the superconductor. 

We have shown that, for any $\lambda_L>0,$ the Debye source formulation provides boundary integral equations, which are Fredholm equations of second-kind, that can be solved for an arbitrary incoming field and fluxes through the $A$-cycles. In numerical examples, we demonstrate the practical efficacy of this method, provided $\lambda_L$ is not too small.  In real applications $\lambda_L$ is typically quite small, and the equations given in this paper will become very difficult to solve. In a second paper, see~\cite{EpRa2}, we show that these equations have limits, in the strong operator sense, as $\lambda_L\to 0^+.$ The limiting system of equations is nearly decoupled and consists of Fredholm equations, some of first kind and some of second kind. These equations can easily be solved and will provide an effective means to accurately approximate the solutions to~\eqref{eqn1},~\eqref{eqn3},~\eqref{eqn_bc} in physically reasonable cases.

\section{Acknowledgements}
The authors would like to thank Antoine Cerfon, Leslie Greengard and, Michael O'Neil for many useful discussions. CLE would also like to thank the Flatiron Institute for hosting him during the 2020-21 academic year, during which most of this research was done, and for their generous, ongoing support.

{\small
  \bibliographystyle{siam}

}

\end{document}